\documentclass[11pt]{amsart}

\usepackage{epigraph}
\usepackage[margin=1.5in]{geometry}
\usepackage{amsmath}
\usepackage{amsfonts}
\usepackage{amssymb}
\usepackage{xcolor}
\def \F {{\mathcal F}}
\def \P {{\mathbb P}}
\def \C {{\mathbb C}}
\def\E{\mathbb E}
\def \R {{\mathbb R}}
\def \N {{\mathbb N}}
\def \Z {{\mathbb Z}}

\newcommand{\PP}{\mathcal P}

\usepackage[french,english]{babel}

\begin{document}

 \newtheorem{theo}{Theorem}[section]
\newtheorem{lemm}[theo]{Lemma} 
\newtheorem{prop}[theo]{Proposition}
\newtheorem{coro}[theo]{Corollary}
\newtheorem{defi}[theo]{Definition}
\newtheorem{rema}[theo]{Remark}
\newtheorem{nota}[theo]{Notation}
\newtheorem*{theo*}{Theorem}

\newenvironment{theobis}[1]
  {\renewcommand{\thetheo}{\ref{#1}$\ bis $}%
   \addtocounter{theo}{-1}%
   \begin{theo}}
  {\end{theo}}

\setcounter{tocdepth}{2}

\numberwithin{equation}{section}
\title[Conditioned Brownian motion in the interval {[0,1]}]{Pitman transforms and Brownian motion in the interval viewed as an affine alcove}

\author{Philippe Bougerol and Manon Defosseux}
\address{Ph.B. : Sorbonne Universit\'e, LPSM-UMR 8001, Paris, France}\email{philippe.bougerol@upmc.fr}
\address{M. D. : Universit\'e Paris Descartes, MAP5-UMR 8145, Paris, France}\email{manon.defosseux@parisdescartes.fr}

\maketitle
   \begin{abstract} Pitman's theorem states that if $\{B_t,t\ge 0\}$ is a one dimensional Brownian motion, then $\{B_t - 2 \inf_{0 \leq s\leq t}B_s, t\ge 0\}$ is a three dimensional Bessel process, i.e.\ a Brownian motion conditioned to remain forever positive.  This paper gives a similar representation for the Brownian motion conditioned to remain in a given  interval.  Due to the double barrier condition, this representation is more involved and only asymptotic. One uses the fact that the interval is an alcove of the Kac-Moody affine Lie algebra $A_1^{(1)}$, the Littelmann path approach of representation theory and a dihedral approximation.
  \end{abstract}
\selectlanguage{french}

 \begin{abstract}
  Le th\'eor\`eme de Pitman affirme que si $\{B_t,t\ge 0\}$ est un mouvement brownien  unidimensionnel, alors $\{B_t - 2 \inf_{0\leq s\leq t}B_s, t\ge 0\}$ est un processus de Bessel de dimension trois, c'est-\`a-dire un brownien  conditionn\'e \`a rester positif. Nous donnons dans cet article une repr\'esentation analogue pour le brownien conditionn\'e \`a rester dans un intervalle donn\'e. En raison de la pr\'esence de deux extr\'emit\'es, cette repr\'esentation est plus compliqu\'ee que celle du th\'eor\`eme original.   Nous utilisons le fait que l'intervalle est une alc\^ove pour l'alg\`ebre de Kac-Moody affine $A_1^{(1)}$, l'approche par le mod\`ele de chemins de Littelmann de la th\'eorie des repr\'esentations et une approximation di\'edrale.
\end{abstract}

\selectlanguage{english}

\setlength{\epigraphwidth}{0.24\textwidth}
  \epigraph{couleurs de la corde

d\'ep\^ot de cette image

cristaux du temps

traces d'espace}{Raymond Queneau} 

\tableofcontents
 
\section{Introduction}

\noindent {\bf 1.1.}
 The probability transition of the Brownian motion  conditioned to stay positive forever is the Doob transform of the difference of two heat kernels. This is a consequence of the reflection principle at $0$. Pitman's theorem \cite{pitman} of 1975 gives the path representation of this process as 
 $$\mathcal PB(t)=B_t-2\inf_{0\leq s \leq t}B_s,$$ where $B$ is a standard Brownian motion with $B_0=0$.
 The transform $\mathcal PB$ is written with the reflection at $0$.  Consider now a Brownian motion conditioned to stay in the interval $[0,1]$ forever.  Is it possible to write it as a path transform of a Brownian motion $B$ by some kind of folding? The conditioned process can be seen as the Doob transform of the Brownian motion killed  at $0$ and at $1$.  Its probability transition is an alternating infinite sum which can be obtained by applying successive reflection principles at 0 and at  1 (method of images). 
 It is therefore natural to ask if Pitman's theorem has an analogue for the conditioned process in the interval written with two similar transforms at 0 and 1, maybe repeated an infinite number of times.  The main result of this article is to show that, to our surprise, this is not exactly the case. A small correction (a L\'evy transform) has to be added: this is due to the non differentiabilty of the Brownian motion. Interestingly, the same correction also occurs in an asymptotic property of the highest weight representations of the affine Lie algebra $A_1^{(1)}$. Hence there are deep links between the trajectories of the Brownian motion and representation theory of Kac Moody algebras.
 \medskip
 
\noindent {\bf 1.2.} Let us state our main result.  We suppose that $\mu\in [0,1]$. We will give in Section \ref{sec_Z} a precise definition of the following process.  \begin{nota} \label{defZ}
$Z^\mu$ is a Brownian motion conditioned to stay in $[0,1]$ forever such that $Z^\mu_0=\mu$.
\end{nota}

We consider, for a continuous real path $\varphi:\R_+\to \R$ such that $\varphi(0)=0$, for $t \geq 0$,
\begin{align*}\mathcal L_1\varphi(t)&=\varphi(t)-\inf_{0\leq s\le t}\varphi(s),\\
\PP_1\varphi(t)&=\varphi(t)-2\inf_{0\leq s\le t}\varphi(s). \end{align*}
We call them the classical L\'evy and  Pitman transforms of  $\varphi$. 
 We introduce
\begin{align*}\mathcal L_0\varphi(t)&=\varphi(t)+\inf_{0\leq s\le t}(s-\varphi(s)),\\
\PP_0\varphi(t)&=\varphi(t)+2\inf_{0\leq s\le t}(s-\varphi(s)). \end{align*}
For $n\in \N$, 
we let $\mathcal P_{2n}=\mathcal P_0,\mathcal L_{2n}=\mathcal L_0$ and $\mathcal P_{2n+1}=\mathcal P_1,\mathcal L_{2n+1}=\mathcal L_1$.  The aim of this paper is  the following representation theorem (see Theorem \ref{thmprincipal}).

\begin{theo*} 
Let $\mu\in [0,1]$ and let $B^\mu_t=B_t+t\mu$ be the real  Brownian motion with drift $\mu$ starting from $0$. For any $t> 0$, almost surely,
$$\lim_{n\to \infty}   t\mathcal L_{n+1} \PP_{n}\cdots \PP_1 \PP_0 B^{\mu}(1/t)=\lim_{n\to \infty}   t\mathcal L_{n +1}\PP_{n}\cdots \PP_2 \PP_1 B^{\mu}(1/t)=Z^\mu_t.$$ 
\end{theo*}

\noindent {\bf 1.3.} Briefly, the strategy of the proof is as follows.
 The first step is to linearize the problem (the reflection at 1 is not linear): 
we introduce the process $A^{(\mu)}$ which is the space-time Brownian motion $B^{(\mu)}_t=(t,B_t^\mu), t \geq 0, $ conditioned to stay in the affine cone $$C_{\mbox{aff}}=\{(t,x)\in \R^2; 0 < x < t\}.$$
We show that  $Z^\mu$ is the space component of the time inverted process of $A^{(\mu)}$, i.e. $A^{(\mu)}_t=(t,t Z^\mu_{1/t})$ in distribution. So we will work essentially with $A^{(\mu)}$.

We define a sequence of  non-negative random processes
$\xi_n(t), t \geq 0, n \in \N$, by
\begin{align}\label{qusi}\xi_n(t)=-\inf_{0 \leq s \le t}\{s1_{2 \N}(n)+(-1)^{n-1}\PP_{n-1}\cdots \PP_0B^{\mu}(s)\}.\end{align}
 Then
 \begin{align}\label{mar}\mathcal P_{n}\cdots\mathcal P_{0}B^{\mu}(t)=B^{\mu}_t+2\sum_{k=0}^n(-1)^{k+1}\xi_k(t),\end{align}
and   
 \begin{align}\label{ser}\mathcal  L_{n+1}\mathcal P_{n}\cdots\mathcal P_{0}B^{\mu}(t)=\mathcal P_{n}\cdots\mathcal P_{0}B^{\mu}(t)+(-1)^n\xi_{n+1}(t).\end{align}
  We first suppose that $\mu\ne 0,1$. Then the random variables $$\xi_n(\infty)=\lim_{t \to +\infty}\xi_n(t)$$ are finite a.s.\ and their distributions have a simple explicit representation with independent exponential random variables. The properties of $\xi_n(t), n\in \N$, can be deduced by conditioning arguments.  This allows us to show that for all $t \geq 0$  the limit of (\ref{ser}) 
exists a.s.\ and has the distribution of the space component of $A^{(\mu)}$. We also prove that for $t>0$, $\xi_n(t)$ tends to  $2$ almost surely when $n$ tends to  $+\infty$. This shows that the limit of (\ref{mar}) itself does not exist. The boundary cases $\mu=0,1,$ are dealt with using the Cameron--Martin--Girsanov (CMG) theorem.

To prove these results, we approximate the space-time Brownian motion $B^{(\mu)}$ by planar Brownian motions with proper drifts and we approximate $A^{(\mu)}$ by these planar Brownian motions conditioned to remain in a wedge in $\R^2$ of dihedral angle ${\pi}/{m}$. The dihedral case has been dealt with in Biane et al.\ \cite{bbo} and we use their results. Due to the need of the correction term, the approximation is not immediate.    
  
  \medskip

\noindent {\bf 1.4.} The article is organized as follows. We always suppose that $0 \leq \mu \leq 1$. In Section  \ref{sec_cond} we first define  rigorously $Z^\mu$. Then we define $A^{(\mu)}$ the conditioned space-time Brownian motion with drift $\mu$ in the affine Weyl cone $C_{\mbox{aff}} $.  We prove in Theorem \ref{bigthm2} that, as processes, $A^{(\mu)}_t=(t,t Z^\mu_{1/t})$ in distribution. 
In Section \ref{Pitmandihedral} we recall the Pitman representation theorem for planar Brownian motions in a dihedral cone and show how they approximate $A^{(\mu)}$. In Section \ref{sec_str_dihe} we introduce the string parameters in the dihedral case. The analogous parameters $\{\xi_n(t), n\in \N, t \geq 0\}$ for the space time Brownian motion $B^{(\mu)}$ are defined in Section \ref{DihtoAff} and called the affine string parameters. When $\mu\ne 0,1,$ then $\xi_n(\infty),n \in \N,$ are finite and are called the Verma affine parameters. 
In Section \ref{HWP} we study the highest weight process $\Lambda^{(\mu)}$ which is the limit of the image of $B^{(\mu)}$ under the transformation (\ref{ser}). It is shown in 
Section \ref{LP} that $\Lambda^{(\mu)}$ equals  $A^{(\mu)}$ in distribution. The representation theorem for the Brownian motion $Z^\mu$ in $[0,1]$ follows.  In Section  \ref{sec_distributions} we first compute 
 the conditional distribution of $B^{(\mu)}_t$ given the sigma-algebra $\sigma\{\Lambda^{(\mu)}(s),s\le t\}$ and then the distributions of $L^{(\mu)}(\infty)$ and  $\xi_1(\infty)$. Up to this point we only use probabilistic arguments with no reference to algebra.

In Section \ref{affine}, we  introduce the infinite dimensional affine Lie algebra $A_1^{(1)}$ and show how our results are related to its highest weight representations.
We show that the conditional distribution of the Brownian motion is a Duistermaat Heckman measure for a circle action. It describes the semiclassical behaviour of the weights of a representation when its highest weight is large. The L\'evy correction term also occurs in the behaviour of the elements of large weight of the Kashiwara crystal $B(\infty)$, which is of independent interest.

  \medskip

\noindent {\bf 1.5.} We have chosen to present the proof of our probabilistic results without explicit reference to Kac-Moody algebra, so that it can be read by probabilists. But let us now explain the ideas of representation theory behind the scenes because this has been a source of inspiration. This may be helpful for some readers. This also explain our choice of terminolgy (affine cone, string parameter, highest weight, ...).

 At the heart of our approach is the fact that  the interval $[0,1]$ is an alcove for the Kac-Moody affine Lie algebra $A_1^{(1)}$ and  that $A^{(\mu)}$ can be seen as a process conditioned to remain in a Weyl chamber. This is linked to highest weight representations of $A_1^{(1)}$ through Littelmann path approach (see \cite{littel}).

To give some details, let us first recall the link between Littelmann path theory for the Lie algebra $\mathfrak{sl}_2(\C)$  and the classical Pitman theorem as explained in  Biane et al.\ \cite{bbo}.  Consider a real line $V=\R \alpha$ where $\alpha$ is the positive root.
A path $\eta$ in $V$ is a continuous function $\eta: \R^+ \to V$ such that $\eta(0)=0$. It can be written as $\eta(s)=\varphi(s)\alpha$ with $\varphi(s)\in \R$, for $s \geq 0$.   A dominant path is a path with values in the Weyl chamber, which is here $\R_+\alpha$, so that $\eta$ is dominant when $\varphi(s) \geq 0$ for all $s \geq 0$. We fix some $t >0$. An integral path on $[0,t]$ is a piecewise linear path such that $2\varphi(t)$ and $2\min_{0\leq s\le t}\varphi(s)$ are in $\Z$. 
Consider an irreducible highest  weight  $\mathfrak{sl}_2(\C)$-module with highest weight $\lambda$. It has  a combinatorial description given by a Kashiwara crystal  (see  Kashiwara \cite{kash93}). Littelmann gives a path realization of this crystal (see Littelmann \cite{littel}). One chooses a dominant integral  path $\pi$ on $[0,t]$, such that $\pi(t)=\lambda$. The  Littelmann module is the set of integral paths $\eta$  on $[0,t]$  such that $\mathcal P_\alpha\eta(s)=\pi(s)$ for all $s\in [0,t]$, where  $\mathcal P_\alpha$ is the path transform defined by 
\begin{align}\label{PT} \mathcal P_\alpha\eta(s)=(\varphi(s)-2\inf_{0 \leq u\le s}\varphi(u))\alpha, \end{align}
when $\eta(s)=\varphi(s)\alpha$. One recognizes the Pitman transform.
 
 Let us consider now the Littelmann path theory for the affine Lie algebra $A_1^{(1)}$. This Kac Moody algebra  corresponds to the group of loops of $Sl(2,\C)$.
 Let $V=\R \Lambda_0 \oplus \R \Lambda_1  \oplus \R \delta$
be a 3 dimensional real vector space with basis $\Lambda_0,  \Lambda_1,  \delta$. We let $\alpha_0=\Lambda_0-\Lambda_1+\delta, \alpha_1=\Lambda_1-\Lambda_0$. In the dual space $\tilde V$ we define $\tilde  \alpha_0, \tilde  \alpha_1$ by requiring that  $\tilde \alpha_0(\delta)=\tilde \alpha_1(\delta)=0$ and that
$( \tilde \alpha_i(\alpha_j))_{0\leq i ,j \leq 1}$ is the Cartan matrix
$$\begin{pmatrix} 2 & -2 \\ -2 & 2 \end{pmatrix}.$$
The Weyl group $W$ is generated by the linear reflections $s_{0},s_{1}$ on $V$ defined  by
$$s_{i}(v)=v- \tilde\alpha_i(v)\alpha_i,$$  for $v \in V$ and $i\in\{0,1\}$. They are reflections along the walls of the Weyl chamber 
$$C_W=\{t\Lambda_0+x\Lambda_1+y \delta, 0 < x < t, y \in \R\}\simeq C_{\mbox{aff}} \times \R.$$
In Littelmann's theory, a path  is now a continuous map $\eta:\R^+\to V$ such that  $\eta(0)=0$.   One defines path transforms $\mathcal P_{s_i}$, $i\in\{0,1\}$, by $$\mathcal P_{s_i}\eta(t)=\eta(t)- \inf_{0 \leq s\le t}\tilde\alpha_i(\eta(s))\alpha_i.$$ 
A dominant path is a path with values in the closure of $ C_W$ and  one can define integral paths on $[0,t]$. For a fixed $t >0$ and  an integral dominant path $\pi$ on $[0,t]$, the Littelmann   module  generated by $\pi$ is the set of  integral paths $\eta $  on $[0,t]$ for which there exists $n\in \N$ such that $$\mathcal P_{s_n}\mathcal P_{s_{n-1}}\cdots \mathcal P_{s_1}\mathcal P_{s_0}\eta(s) = \pi(s)$$ when $0 \leq s \leq t,$ where  $s_{2k}=s_{0}$ and $s_{2k+1}=s_1$ for $k\in \N$. This gives a description of the Kashiwara crystal of highest weight $\pi(t)$ (see Kashiwara \cite{kash95}, Littelmann \cite{littel}).
For an integral path $\eta$ on $[0,t]$ and more generally  for a continuous piecewise $C^1$ path, there is a $k\in \N$ such that
 for all $n\ge k$, 
$$\mathcal P_{s_n}\cdots\mathcal P_{s_0}\eta(s)=\mathcal P_{s_k}\cdots\mathcal P_{s_0}\eta(s),$$ 
for  $0 \leq s \leq t$ (see Proposition \ref{nombrefini}). This new path is dominant.
One can ask if, similarly to Pitman's theorem, at least the limit, when $n$ tends to infinity,
of $\mathcal P_{s_n}\cdots\mathcal P_{s_0}\eta$
exists when $\eta $ is replaced by a space-time Brownian motion.

We show in this paper that this is not the case, but a slight modification does converge in $V/\R\delta$ to a space-time Brownian motion  conditioned to remain in the affine cone $C_{\mbox{aff}} $.  It will be enough for us to work in the space $V/\R\delta$, so we will use a space $V$ without $\delta$.

 We call the sequence $\{\xi_k(t), k\in \N\}$ defined in (\ref{qusi}) the affine string parameters by analogy with the string parameters in the Littelmann model. For integral paths,  they are the string parameters of the corresponding element  in the highest weight crystal (see Kashiwara \cite{kash93}, Littelmann \cite{littel2}).
Likewise, we call  $\xi_k(\infty), k \in \N,$ the Verma affine string parameters, by analogy with the string parameters of the crystal $B(\infty)$ of Kashiwara associated with the Verma module of highest weight $0$. The fact that we are able to study the affine string parameters from the Verma ones is reminiscent of the fact that the highest weight crystals are obtained from $B(\infty)$.

   \medskip

\noindent {\bf 1.6.} In conclusion we see that the  conditioned Brownian motion in the interval, a priori simple, can be studied here thanks to its links with the infinite dimensional affine Lie algebra $A_1^{(1)}$. It is an example of integrable probability in the sense of Borodin and Petrov \cite{boropetrov}. It is interesting to remark that we use an approximation by processes which has no direct relation with group theoretic diffusions, since the dihedral group is not the Weyl group of a semi-simple Lie algebra (when $m \geq 7$).  It would be interesting to study the higher rank case $A_n^{(1)}, n \geq 2,$ which occurs in the analysis of $n+1$ non colliding Brownian motions on a circle (see Hobson and Werner \cite{hobwern}) or of the eigenvalues of the Brownian motion on $SU(n+1)$. This requires new ideas.
\newline

\noindent {\bf 1.7.} We always suppose that $0 \leq \mu \leq 1$.
\newline

\noindent {\bf 1.8.} We thank Philippe Biane,  Persi Diaconis and the referees for their advice.

\section{Conditioned Brownian motions}\label{sec_cond}
\subsection{The conditioned Brownian motion in $[0,1]$}\label{sec_Z}

We recall some known facts about the Brownian motion conditioned to stay forever in the interval $[0,1]$. It was first introduced by Knight \cite{knig} when the starting point is in $(0,1)$ and called the {\it taboo process}. It was defined there as the limit when $t $ tends to infinity of a standard Brownian motion starting in the interval, conditioned to reach the boundary after time $t$. To define it rigorously in general, let us first consider the Brownian motion in $(0,1)$ killed at the boundary.  Its generator is $\frac{1}{2}\frac{d^2}{dx^2}$  with Dirichlet boundary conditions. Its maximal eigenvalue  is $-\pi^2/2$ with positive eigenvector 
$$h(x)=\sin (\pi x),$$
called the ground state.  We consider 
the associated $h$-Doob process $\{Z_t, t \geq0\}$.
It is the Markov process with transition probability density $q_t(x,y)$  given by 
\begin{align}\label{q_t}
q_t(x,y)
=\frac{\sin(\pi y)}{\sin(\pi x)}e^{\pi^2 t/2}u_t(x,y),
\end{align}
for $x,y \in (0,1)$, where $u_t(x,y)$ is the transition probability density of the killed Brownian motion.
 It
is the diffusion in $[0,1]$, with infinitesimal generator given by 
\begin{align}\label{EDSZ}\frac{1}{2}\frac{d^2}{dx^2}+\pi \cot(\pi x) \frac{d}{dx}.\end{align}
By the reflection principle
\begin{align}\label{refl01}
u_t(x,y)=\sum_{k\in \Z}(p_t(x,y+2k)-p_t(x,-y+2k)),\end{align}
for $x,y \in (0,1)$, where $p_t$ is the  standard heat kernel (see  Ito and McKean \cite{ItoMcKean}, p.30). 
Using for instance the Poisson formula (see Bellman \cite{Bellman}) one has,
\begin{align}\label{refl02}
u_t(x,y)=\sum_{n\in \N}\sin(n\pi x)\sin(n\pi y)e^{-\pi^2 n^2 t/2}.
\end{align}
One sees that $0$ and $1$ are entrance non--exit boundaries, by scale function techniques for instance. In other words, $(Z_t)$ can start from the boundaries $0$ and $1$ and does not touch them at positive time. Let us remark that $Z$ can also be defined by the latitude of the Brownian motion on the $3$-dimensional sphere (see Ito and McKean \cite{ItoMcKean}, Section 7.15) or by the argument of an eigenvalue of the Brownian motion in $SU(2)$. The behaviour at the boundaries is also clear from these descriptions.
When $Z_0=0$, the entrance density measure starting  from $0$ is the limit of $q_t(x,y)$ when $x$ tends to $ 0$; hence,  by (\ref{refl02})
\begin{align}\label{entrance}q_t(0,y)=\sin(\pi y)\sum_{n\in \Z}n\sin(n\pi y)e^{-\frac{t}{2}\pi^2(n^2-1)}, \end{align}
for any $y\in (0,1)$. The case $Z_0=1$ is symmetric. 

For $\mu\in [0,1]$, we write $(Z^\mu_t)$ for the process $Z$ when $Z_0=\mu.$
and we call it the conditioned Brownian motion process in $[0,1]$.

\subsection{The conditioned space-time Brownian motion in $C_{\mbox{aff}}$ }\label{AppAffi} Let $(B_t)$ be the standard real Brownian motion starting from $0$ and $$ B_t^{(\mu)}=(t,B_t+t\mu)$$ be the   space-time Brownian motion with drift $\mu$.
We now define rigorously the process $ A^{(\mu)}(t), t \geq 0,$ which is the process  $B^{(\mu)}$ conditioned to stay forever in the affine cone $$C_{\mbox{aff}}=\{(t,x)\in\R_+\times \R_+ : 0<x<t \},$$ starting from $(0,0)$. It has been introduced and studied in Defosseux \cite{ defo2, defo3}, or also \cite{ defo1}.
Let $\{K_t^{(\mu)}, t \geq 0\}$ be the space-time process $ B_t^{(\mu)}$ killed at the boundary of $\bar { C}_{\mbox{aff}}$. This is the process in the cone with generator $\frac{\partial}{\partial t}+\frac{1}{2}\frac{\partial^2}{\partial x^2}+\mu \frac{\partial}{\partial x}$ and Dirichlet boundary conditions. 

We define a theta function $\varphi_\mu$ on $\R_+^*\times \R$ first when $\mu\ne 0,1,$  by 
\begin{align}\label{defphi}\varphi_\mu(t,x)=
\frac{e^{-\mu x}}{\sin(\mu\pi )}
\sum_{k\in \Z}{ \sinh(\mu(2kt+x))}e^{-2(kx+k^2t)}, 
\end{align} 
for  $t>0, x\in \R.$  
It follows from (\ref{refl01}) that  \begin{align}\label{Poisson}
\varphi_\mu(t,x)=\frac{\sqrt{\pi}}{\sqrt{2t}\sin(\mu \pi)}u_{1/t}(\mu,x/t)e^{\frac{t}{2}(\mu-x/t)^2},
\end{align}
which implies by (\ref{refl02}) that $\varphi_{\mu}(t,x)=\varphi_{1-\mu}(t,t-x)$.
By continuity,
$$
\varphi_0(t,x)=\varphi_1(t,t-x)=\frac{1}{\pi}\sum_{k\in \Z} (2kt+x)e^{-2(kx+k^2t)}.$$

\begin{prop}\label{harmo}
The function $\varphi_\mu$  is  a space-time non negative harmonic function for the killed process $K^{(\mu)}$ on $C_{\mbox{aff}}$,  vanishing on its boundary.
\end{prop}

\begin{proof} The fact that $\varphi_\mu$ is harmonic is clear by computation. 
The boundary condition $\varphi_\mu(t,0)=0$, resp.\   $ \varphi_\mu(t,t)=0$, follows from the change of variable from $k$ to $-k$, resp.\ $k$ to $-1-k$. It is non negative by (\ref{Poisson}).\end{proof}

Let us  write $K^{(\mu)}_t=(t, K^\mu_t)$ and let $w^{\mu}_t((r,x),(t+r,.))$ be the density  of $K_{t+r}^\mu$ given $K_r^\mu=x$.  Using the Cameron-Martin-Girsanov formula (see \cite{kara}, Theorem 3.5.1) one has 
\begin{align}\label{egalw}w^{\mu}_t((r,x),(t+r,y))=e^{\mu(y-x)-\frac{t}{2}\mu^2}w^{0}_t((r,x),(t+r,y))\end{align}
and, by Defosseux (\cite{defo2}, Proposition 2.2),  
\begin{align*}
w^{\mu}_t&((r,x),(t+r,y))\\
&=\sum_{k\in \Z}e^{-2(kx+k^2r)}(e^{\mu(2kr+x)}p^\mu_t(x+2kr,y) 
-e^{-\mu(2kr+x)}p^\mu_t(-x-2kr,y)),\end{align*}
where $p^\mu_t$ is the standard heat kernel with  drift $\mu$. We define the Markov process  $\{ A^{(\mu)}(t), t \geq 0\}$ in $ C_{\mbox{aff}}\cup \{(0,0)\}$ as the Doob transform of $K^{(\mu)}$ in the following way:  $A^{(\mu)}(t)=(t,A^\mu_t)$ where the distribution of $A_{r+t}^\mu$ given $A^\mu_r=x $ has the density 
\begin{align} \label{transdensi}s^\mu_t((r,x),(r+t,y))=\frac{\varphi_\mu(r+t,y)}{\varphi_\mu(r,x)}  w^{\mu}_t((r,x),(r+t,y)),\end{align}
for $(r,x), (r+t,y)\in C_{\mbox{aff}}$,  $A^{(\mu)}(0)=(0,0),$  and the
 entrance density is given by 
\begin{align} \label{entrancedrift}
s^\mu_t((0,0),(t,y))=\varphi_\mu(t,y)\sin(\frac{y}{t}\pi)e^{-\frac{1}{2t}(y-\mu t)^2}, \end{align}
for $ (t,y)\in C_{\mbox{aff}}$.

\begin{defi}\label{defi_A} The process $\{ A^{(\mu)}(t), t \geq 0\}$
 is the space-time Brownian motion conditioned to stay in $C_{\mbox{aff}}$.
\end{defi}

Recall that $u_t$, resp.\ $q_t$, is the transition probability density of the killed Brownian motion in $[ 0,1]$, resp.\ of $Z^\mu$ (see (\ref{q_t}) and (\ref{refl01})).
\begin{lemm}\label{ppchapeau}   For $0 <x \leq r\leq t$ and $0 \leq y \leq t,$
$$u_{1/r-1/t}(y/t,x/r)e^{-\frac{1}{2t}y^2}=w_{t-r}^0((r,x),(t,y))e^{-\frac{1}{2r}x^2},$$
and for $0 \leq x,y \leq 1$,
$$q_t(x,y)=\frac{1}{\sqrt{2\pi t}}e^{ \pi^2t/2}\sin(\pi y)e^{-\frac{1}{2t}(y-x)^2}\varphi_x(1/t,y/t). 
$$
\end{lemm}

\begin{proof}
\begin{align*}
u_{1/r-1/t}(y/t,x/r)=&\sum_{k\in \Z}(p_{1/r-1/t}(y/t+2k,x/r)-p_{1/r-1/t}(-y/t-2k,x/r))\\
=&\sum_{k\in \Z}\frac{p_r(0,x)}{p_t(0,y+2kt)}(p_{t-r}(x,y+2kt)-p_{t-r}(x,-y-2kt))\\
=&\frac{p_r(0,x)}{p_t(0,y)}\sum_{k\in \Z}e^{-2kx-2k^2r}(p_{t-r}(x+2kt,y)-p_{t-r}(-x-2kr,y)),
\end{align*}
 gives the first equality. The second one follows from   (\ref{Poisson}) and  (\ref{q_t}).
\end{proof}
\begin{theo} \label{bigthm2} The processes $\{A^{(\mu)}(t),t \geq  0\}$ and $\{(t,tZ^\mu_{1/t}), t\geq 0 \}$  have the same distribution.
\end{theo}
\begin{proof}
For $0<t_1<\cdots<t_n$, $\P(t_1Z^\mu_{1/t_1}\in dx_1, \cdots, t_nZ^\mu_{1/t_n}\in dx_n)$ equals  $$q_{\frac{1}{t_n}}(\mu,\frac{x_n}{t_n})q_{\frac{1}{t_{n-1}}-\frac{1}{t_n}}(\frac{x_n}{t_n},\frac{x_{n-1}}{t_{n-1}})\cdots q_{\frac{1}{t_1}-\frac{1}{t_2}}(\frac{x_2}{t_2},\frac{x_{1}}{t_{1}})\,dx_1\cdots dx_n .$$
Identity (\ref{q_t}) and the first identity of Lemma  \ref{ppchapeau}    imply that
\begin{align*}
\prod_{k=2}^n&q_{\frac{1}{t_{k-1}}-\frac{1}{t_k}}(\frac{x_k}{t_k},\frac{x_{k-1}}{t_{k-1}})\\
&=ce^{\frac{1}{2t_n}x_n^2-\frac{1}{2t_{1}}x_{1}^2}\frac{\sin(\pi x_1/t_1)}{\sin( \pi x_n/t_n)}\prod_{k=2}^n w^0_{t_k-t_{k-1}}((t_{k-1},x_{k-1}),(t_k,x_k)),
\end{align*}
where $c$ is independent of $x_1,\cdots, x_n$.We conclude by using the second identity of Lemma  \ref{ppchapeau} for $q_{\frac{1}{t_n}}(\mu,\frac{x_n}{t_n})$ and  (\ref{egalw}),(\ref{transdensi}), and (\ref{entrancedrift}).
\end{proof}
This implies that $A^{(\mu)}(t)$ is really in the interior of the cone for $t >0$.

\section{Approximation by planar Brownian motions in dihedral cones}\label{Pitmandihedral}

In this section we approach the conditioned process $A^{(\mu)}$ by planar Brownian motions conditioned to stay in a dihedral cone. 
\begin{defi}
For $m\in \N^*$, the convex dihedral cone $C_m$, with closure $\bar C_m$, is 
\begin{align*}
C_m=\{(r\cos \theta,r \sin\theta)\in \R^2; r >0, 0 < \theta < \pi/m\}.\end{align*}\end{defi}

\subsection{Conditioned Brownian motion in a dihedral cone}\label{subdihe}
We will use the results of Biane et al.\ \cite{bbo},\cite{bbo2}, that we first recall.
We consider   $V=\R^2$ identified with its dual $\tilde V$, equipped with the usual scalar product $\langle .,. \rangle$.  
The dihedral group $I(m)$ is the finite group generated by two involutions $s_0^m, s_1^m$ with the only relation 
 $(s_0^ms_1^m)^m=1.$ Notice that the exponent $m$ in $s_0^m$ and $s_1^m$ is not a power but only refers to the angle of the cone. One chooses two pairs $(\alpha^m_0,\tilde \alpha^m_0), (\alpha^m_1,\tilde \alpha^m_1)$  in $ V  \times \tilde V$, associated with the matrix $(\tilde \alpha^m_i(\alpha^m_j))_{0\leq i ,j \leq 1}$ given by
 $$
\begin{pmatrix} 
2 & -2\cos(\pi/m)  \\
-2\cos(\pi/m) & 2  \\
\end{pmatrix}. $$
 Namely, one takes 
\begin{align*}   \alpha^m_0=(2\sin( \pi/m), -2\cos( \pi/m)), 
 \alpha^m_1=(0,2), 
\end{align*} 
and $\tilde \alpha^m_0=\alpha^m_0/2, \tilde \alpha^m_1=\alpha^m_1/2.$
 The following two linear reflections $s_0^m,s_1^m$ of $\R^2$,  $$s_i^m(v)=v-\tilde \alpha^m_i(v)\alpha^m_i,\;\; v \in \R^2,$$
 generate the group $I(m)$.    With our convention $\tilde \alpha^m_i(v)=\langle \tilde \alpha^m_i, v\rangle$ for $v\in V$.
 One has 
 \begin{align*}
C_m&=\{v \in \R^2; \tilde \alpha^m_i(v) > 0, i=0,1\},\end{align*}
and $\bar C_m$ is a fundamental domain for the action of $I(m)$ on $\R^2$. Let $\gamma \in \bar C_m$.
The conditioned planar Brownian motion in $C_m$ with drift $\gamma$ is intuitively given by the two-dimensional Brownian motion with drift starting from $0$ conditioned to stay in the cone $C_m$ forever. It is rigorously defined in 5.1 of Biane et al.\ \cite{bbo2} (see Proof of Theorem 5.5). The definition is as follows. Let $\kappa^{(\gamma)}_t$ be the semi group density of the  standard planar Brownian motion in $C_m$ with drift $\gamma$ killed at the boundary.
The function
\begin{align}\label{psigamma}\psi_{\gamma}^m(v)=\sum_{w\in I(m)}(-1)^{\det(w)}e^{\langle w(\gamma)-\gamma,v\rangle},\end{align}
 for $\gamma \in \bar C_m$, $v\in \R^2$,  is harmonic positive in $C_m$, null on the boundary. We introduce the  Doob transform $$r^{(\gamma)}_t(v_1,v_2)=\frac{\psi_\gamma^m(v_2)}{\psi_\gamma^m(v_1)}\kappa_t^{(\gamma)}(v_1,v_2),$$
where $v_1,v_2\in C_m$. Let \begin{align}\label{hm}h_m(v)=\Im ((x+iy)^m)\end{align} for $v=(x,y)\in \R^2$
(this is  the alternating  polynomial associated to the dihedral group  $I(m)$, equal to the product of the roots, see Dunkl and Xu \cite{dunxu}, 6.2.3).

 \begin{defi} \label{defiA} The conditioned planar Brownian motion $A_m^{(\gamma)}(t), t \geq 0,$ in the dihedral cone $C_m$ with drift $\gamma \in \bar C_m$
is defined as follows. It is the continuous Markov process with values in $C_m \cup \{0\}$ with transition probability densities 
$r_t^{(\gamma)}, t \geq 0,$
 such that $A^{(\gamma)}_m(0)=0$, and with entrance probability density at time $t$ proportional to $$h_m(v)\frac{\psi_\gamma^m(v)}{h_m(\gamma)}e^{-\frac{1}{2t}\langle v,v\rangle}{\bf 1}_{C_m}(v).$$
 \end{defi}
Here, $ {\psi_\gamma^m(v)}/{h_m(\gamma)}$ is obtained by analytical continuation when $h_m(\gamma)=0$ (see Lemma 5.11 in \cite{bbo2}). In particular   for $\gamma=0$ the entrance distribution is proportional to 
\begin{align}\label{hzero}h_m^2(v)e^{-\frac{1}{2t}\langle v,v\rangle}{\bf 1}_{C_m}(v)\end{align}
(see (5.1) of \cite{bbo2}).

 Another quick manner to define rigorously $A_m^{(\gamma)}$  is to use Dunkl processes  and the approach given by Gallardo and Yor in \cite{galyor}. One considers the Dunkl Laplacian with multiplicity one associated with $I(m)$. it is given by
$$\tilde \Delta f(v)=\Delta f(v) + 2 \sum_{\alpha \in R^+}(\frac{\langle \alpha, \nabla f(v) \rangle}{\langle\alpha,v \rangle}-\frac{ f(v)-f(\sigma_\alpha v)}{\langle \alpha,v\rangle^2}),$$
for $f \in C^2(\R^2), v\in \R^2$, where $\sigma_\alpha$ is the orthogonal reflection with respect to $\{v\in \R^2, \langle \alpha, v \rangle=0\}$, and where $R^+$ is the set of positive roots. Here  $\Delta$ is the usual Laplacian and $\nabla$ is the usual gradient.
Roesler and Voit  (\cite{roslervoit}, Section 3) show that for any $\gamma \in \R^2$, there exists a c\`ad-l\`ag process $D_\gamma(t), t \geq 0,$ with generator $\tilde \Delta/2$ such that 
$D_\gamma(0)=\gamma$. Let $p:\R^2\to \bar C_m$ be the projection  where $\bar C_m$ is identified with $\R^2/I(m)$. Then 
$ Y_\gamma(t)=p(D_\gamma(t))$ is a continuous process called the radial Dunkl process. Its norm is a Bessel process of dimension $2(m+1)$ (\cite{roslervoit}, Theorems 4.10, 4.11).
According to Gallardo and Yor (\cite{galyor}, Section 3.3),  the conditioned Brownian motion in $C_m$ with drift $\gamma \in \bar C_m$ starting from $0$ is defined as $t Y_\gamma(1/t)$. 
Its probability transitions are given  by (25) of \cite{galyor}, which shows that it coincides in distribution with $A_m^{(\gamma)}$.

\subsection{Pitman representation in a dihedral cone}

Let ${\mathcal C}_0(\R^2)$ be
 the set of continuous path $\eta:\R^+\to \R^2$ such that 
 $\eta(0)=0$.  The following path transforms are introduced in \cite{bbo}.
\begin{defi}\label{pitman-transform}The Pitman
    transforms  $\PP^m_{s_i^m}, i=0,1,$
    are defined on ${\mathcal C}_0(\R^2)$ by
$$\PP^{m}_{s_i^m} \eta(t)=\eta(t)-\inf_{0\le s\le t}\tilde \alpha^m_i(
\eta(s))\alpha^m_i, $$
for $\eta \in {\mathcal C}_0(\R^2)$, $t\ge 0.$
\end{defi}
For each $i\in \N$,  we take $i$ modulo 2 when we write $\alpha^m_i, \tilde \alpha^m_i, s_i^m,\cdots $. Any  $w\in I(m)$ can be written (maybe not uniquely) as 
$$w=s_{i_r}^m\cdots s_{i_1}^m$$ for some $r \leq m$.
It is called reduced when $r$ is as small as possible.

\begin{theo}[Biane et al.\  \cite{bbo}]\label{braidP} Let $w=s_{i_r}^m\cdots s_{i_1}^m$ be a reduced decomposition of $w \in I(m)$ where $i_1,\cdots, i_r\in \{0,1\} $.  Then
$$\PP^m_{w}:=\PP^m_{s_{i_r^m}}\cdots\PP^m_{s_{i_1^m}}$$
depends only on $w$ and not on its reduced decomposition.  
\end{theo}
There is a unique longest element $w_m$ in $I(m)$. It has two reduced decompositions, namely
\begin{align}\label{w0}w_m= s_{m-1}^m\cdots s_1^ms_0^m= s_m^m\cdots s_2^ms_1^m.\end{align}
\begin{prop} [\cite{bbo}]For any path $\eta \in {\mathcal C}_0(\R^2)$, the path 
$\PP^m_{w_m}\eta$ takes values in the closed dihedral cone $\bar   
C_m$. 
\end{prop}

The following important theorem is from Biane et al. \cite{bbo2} when $\gamma=0$.
\begin{theo} \label{bbodrift} Let $W^{(\gamma)}$ be the planar Brownian motion with drift $\gamma\in \bar C_m$, identity covariance matrix, starting from the origin. The process $\mathcal P_{w_m}^mW^{(\gamma)}$  has the same distribution as $A_m^{(\gamma)}$.
\end{theo}
\begin{proof}
For  $0<t_1<\dots<t_n$, and a bounded measurable function $F:\R^n\to\R$,    the Cameron--Martin--Girsanov formula (see \cite{kara}, Theorem 3.5.1) gives 
\begin{align*}
\E(F(\mathcal P_{w_m}^mW^{(\gamma)}(t_1),&\dots,\mathcal P_{w_m}^mW^{(\gamma)}(t_n))\\ & =\E(F(\mathcal P_{w_m}^mW^{(0)}(t_1),\dots,\mathcal P_{w_m}^mW^{(0)}(t_n))e^{-\frac{1}{2}\langle\gamma,\gamma\rangle t_n+\langle\gamma,W^{(0)}_{t_n}\rangle}).
\end{align*}
By Theorems 5.1 and 5.5 of \cite{bbo2} and (\ref{hzero}) this equals  \begin{align*}
C\, \E(F(A_m^{(0)}(t_1),\dots,A_m^{(0)}(t_n))e^{-\frac{1}{2}\langle\gamma,\gamma\rangle t_n} \frac{\psi_\gamma^m(A_m^{(0)}(t_n))}{h_m(\gamma)h_m(A^{(0)}_m(t_n))}e^{\langle\gamma,A_m^{(0)}(t_n)\rangle}),
\end{align*}
where $C$ is a constant. This implies the theorem by using Definition \ref{defiA}.
\end{proof}

\subsection{Approximation of the conditioned space time process} 
 
We consider the space-time process
$$B_t^{(\mu)}=(t,B_t+t \mu)$$ where $B$ is a standard real Brownian motion starting from $0$. Recall that we always suppose that $0 \leq \mu \leq 1$.  For $m\in \N^*$, let $W^{(\frac{m}{\pi},\mu)}$ be the planar Brownian motion with drift $(\frac{m}{\pi},\mu)$. Hence we can write 
 $$	W^{(\frac{m}{\pi},\mu)}(t)=(\beta_t+t\frac{m}{\pi},B_t+t\mu ),$$  where $\beta$ is a standard  Brownian motions independent of $B$. 
Let $\tau_m:\R^2\to \R^2$ be defined by
\begin{align*}\tau_m(t,x)=(\frac{\pi t}{m},x),\end{align*}
for $(t,x)\in \R^2$.

Our first approximation theorem is the following. Recall that $A^{(\mu)}$ is the conditioned space-time Brownian motion in $C_{\mbox{aff}}$. We use the standard definition of convergence in distribution for sequence of continuous processes as in Revuz and Yor (\cite{revuzyor}, XIII.1).

\begin{theo}\label{thmconvTau} In distribution, when $m$ tends to $+\infty$, 

(i) $\tau_mW^{(\frac{m}{\pi},\mu)}$ converges to $B^{(\mu)},$

(ii)
$\tau_m\PP^m_{w_m}W^{(\frac{m}{\pi},\mu)}$ converges to $A^{(\mu)}$.\end{theo}
The first statement is clear. We will prove the second one after the following proposition.
 Let $\{Z^\mu_t, t \geq 0\}$ be the conditioned Brownian motion in $[0,1]$ starting from $\mu$.

\begin{prop} \label{thm_limite} For $m\in \N^*$,  let  $X_t^{(m)}, t \geq 0,$ be the $\R^2$-valued continuous process such that $X_{m^2t}^{(m)}$ is the conditioned planar Brownian motion in the cone  $C_m$, with drift $\gamma_m=m^2e^{i \pi\mu/m}$. One writes in polar coordinates $$X_t^{(m)}=R_t^{(m)} \exp{i \pi\theta^{(m)}_t},$$ $R_t^{(m)} \geq 0, \theta_t^{(m)}\in [0,1/m].$ 
Then the process $m\theta_{t}^{(m)},t > 0,$ tends to $ Z^\mu_{1/\pi^2t}$ and the process $ R_{t}^{(m)}, t\geq 0,$ tends to $t$ in distribution when $m $ tends to $+\infty$.
\end{prop}

\begin{proof}  As seen above in Section \ref{subdihe}, it follows from Gallardo and Yor \cite{galyor}  that $Y_{\gamma_m}(t)=tX^{(m)}_{m^2/t}$  is a radial multidimensional Dunkl process starting from $\gamma_m$. Therefore  $R_{m^2t}^{(m)}$ is a Bessel process of dimension $2(m+1)$ with drift $m^2$, starting from $0$.
In other words, one can write
$$( R_{m^2t}^{(m)})^2=(m^2 t +B_t{(1)})^2+\sum_{k=2}^{2(m+1)} (B_t(k))^2$$
where $B_t(1),\cdots, B_t(2(m+1))$  are independent standard real Brownian motions. Since, for any $t >0$,  $$\E(\sup_{0 \leq s \leq t}\sum_{k=1}^{2(m+1)}B_{s/m^2}(k)^2)\leq 8t(m+1)/m^2$$ tends to 0 as $m $ tends to $ +\infty$, the process
$R_{t}^{(m)},t \geq 0 $ converges to  $t$ in distribution. Using the skew product decomposition of $Y_{\gamma_m}$,  one can write 
$$\pi\theta^{(m)}_{m^2t}=\sigma^{(m)}(a_t^{(m)}),\,\, \mbox{ with }  a_t^{(m)}=\int_t^{+\infty}\frac{1}{(R_{m^2s}^{(m)})^2}\, ds,$$ where the process
$\sigma^{(m)}(t), t \geq 0,$ is a solution of the following stochastic differential equation
$$d\sigma^{(m)}(t)=dB_t+m\cot(m \sigma^{(m)}(t)) dt$$
Here $B$ is a Brownian motion independent of $ R^{(m)}$, and $\sigma_0^{(m)}=\mu\pi/m$ (see Demni \cite{dem}). One remarks that $Z_t=\frac{m}{\pi}\sigma^{(m)}({\pi^2t/m^2})$
satisfies
\begin{align}\label{cont_ini} dZ_t=d\beta_t+\pi\cot(\pi Z_t) dt\end{align} where $\beta$ is another Brownian motion and $Z_0=\mu$.
Therefore  $ Z$ coincides with $Z^\mu$ (see  (\ref{EDSZ})). When $m$ tends to $ +\infty$,
$ a^{(m)}_{t/m^2}$
is equivalent to $1/tm^2$; hence
$m\theta^{(m)}_{t}=\frac{m}{\pi}\sigma^{(m)}({ a^{(m)}_{t/m^2}}), t > 0,$ converges in distribution  to $Z^\mu_{1/t\pi^2}.$ 
\end{proof}

\begin{proof}[Proof of (ii) of Theorem \ref{thmconvTau}]
Let
$$X^{(m)}_t=\frac{1}{m \pi }\PP^m_{w_m}W^{(\frac{m}{\pi},\mu)}(\pi^2t).$$
Since
$\frac{1}{m \pi }W^{(\frac{m}{\pi},\mu)}_{m^2\pi^2 t}$ is a planar Brownian motion with drift  $(m^2 ,  m\pi \mu),$
one sees that $X^{(m)}_{m^2t}$ is the conditioned planar Brownian motion in the cone $C_m$ with a drift equivalent to $m^2e^{i\pi\mu /m}$ (see Theorem \ref{bbodrift}).  One writes its polar decomposition as  $X_t^{(m)}=R_t^{(m)} \exp{i \pi\theta^{(m)}_t}.$ Using the continuity of the solution of (\ref{cont_ini}) with respect to the initial condition, we see that Proposition  \ref {thm_limite} also holds when the drift $\gamma_m$  is only equivalent to $m^2e^{i\pi\mu /m}$. Therefore $m\theta_{t}^{(m)}$ tends to $ Z^\mu_{1/t\pi^2}$ and $R_{t}^{(m)}$ tends to $ t$ in distribution. As a consequence the process
$$\tau_m\PP^m_{w_m}W^{(\frac{m}{\pi},\mu)}(t)=({\pi^2}R^{(m)}_{t/\pi^2}\cos(\pi\theta^{(m)}_{t/\pi^2}), m\pi R^{(m)}_{t/\pi^2}\sin(\pi\theta^{(m)}_{t/\pi^2})), \, t \geq 0,$$
converges in distribution to $(t, tZ^\mu_{1/t})$ which is equal  in distribution to $A^{(\mu)}$ by Theorem \ref{bigthm2}.
\end{proof}

\section{String parameters in the dihedral case}\label{sec_str_dihe}
For simplicity of notations, and without loss of generality, one chooses one of the two decompositions of the longest element $w_m$ in the dihedral group $I(m)$, namely,
$$w_m= s_{m-1}^m\cdots s_1^ms_0^m.$$
For $\eta \in {\mathcal C}_0(\R^2)$, $0 \leq k \leq m-1$ and $0 \leq t \leq +\infty$, let $$x_k^m(t)=-
    \inf_{0\leq s\leq t}\tilde \alpha^m_k( \PP^m_{s^m_{{k-1}}}\ldots \PP^m_{s^m_{1}} \PP^m_{s^m_{0}}\eta(s)).$$
We call $$x^m(t)=(x_0^m(t),x_1^m(t),\cdots,x_{m-1}^m(t))$$  the string parameters of the path $\eta$ on $[0,t]$ and we call $x^m(\infty)$ its Verma string parameters.    One has, when $t<+\infty$,
\begin{align}\label{P_eta}
\PP^m_{w_m} \eta(t)=\eta(t)+\sum_{k=0}^{m-1}x_k^{m}(t) \alpha^m_{k}.
\end{align} 

\begin{prop}\label{x_fini}
When
\begin{align}\label{etainfini}
 \lim_{s \to +\infty} \tilde \alpha^m_0(\eta(s))=\lim_{s \to +\infty} \tilde \alpha^m_1(\eta(s))=+\infty,  \end{align} then for all $k$, $x_k^m(\infty)< +\infty$ and   $x_k^m(t)= x_k^m(\infty)$ for $t$ large enough.
\end{prop}

\begin{proof} Let
$\mathfrak t_{0}=\max\{t \geq 0,  \eta(t)\not \in  C_m\},$ and, for $n \geq 0$,
$$\mathfrak t_{n+1}=\max\{t \geq \mathfrak t_n,  \eta(t)+\sum_{k=0}^n x_k^m(t) \alpha_k \not \in  C_{m}\}. $$
 We see by induction that  $\mathfrak t_n < +\infty$ and that for $t > \mathfrak t_n$, $x_n^m(t)=x_n^m(\infty)$.\end{proof}

Let, for $1 \leq k < m$, $$a_k^m=\sin(k\pi/m).$$
\begin{defi} \label{definitionGamma_m}The cone $\Gamma_m$ in $\R^m$ is defined as 
$$\Gamma_m=\{(x_0,\cdots, x_{m-1})\in \R^m;  \frac{x_1}{a^m_1} \geq \frac{x_2}{a^m_2} \geq \cdots  \geq 
\frac{x_{m-1}}{a^m_{m-1}} \geq 0,\,  x_0 \geq 0\}.$$
For $\lambda \in \bar C_m$, the polytope $\Gamma_m(\lambda)$ is
    $$\Gamma_m(\lambda)=\{(x_0,\cdots,x_{m-1})\in \Gamma_m; 0 \leq x_r \leq \tilde \alpha^m_{r}(\lambda-\sum_{n=r+1}^{m-1}x_n \alpha^m_{n}), 0 \leq r \leq m-1\}.$$
    \end{defi}  
  Remark the particular role of $x_0$.  The following proposition is proved in \cite{bbo2}, Propositions 4.4 and 4.7.
  \begin{prop}For $t>0$, 
  the set of string parameters of  paths on $[0,t]$  is   $\Gamma_m$. \end{prop}
  
     Let $W^{(\gamma)}$ be a planar Brownian motion with drift $\gamma$ in $\bar C_m$ starting from 0.
     
\begin{defi} We define $\xi^m(t)$ as the string parameters of $W^{(\gamma)}$ on $[0,t]$ and  $\xi^m(\infty)$ as its Verma string parameters.
\end{defi}

Notice that, a.s.,  when $\gamma \in C_m$, $W^{(\gamma)}$  satisfies  (\ref{etainfini}) by the law of large numbers, since $W^{(\gamma)}_t/t$ converges to an element of $C_m$ when $t\to +\infty$.  So its Verma string parameters are finite by Proposition \ref{x_fini}. We will need the following proposition.

\begin{prop} \label{prop_bayes} For $\nu=0$ and $\gamma$, let
 $\F_t^\nu=\sigma(\PP^m_{w_m} W^{(\nu)}(s), s \leq t)$ and $\psi:{\mathcal C}_0([0,t],\R^2)\to \R$
 be a bounded measurable function. There exists a function $\varphi:\R^2\to \R$ such that, for both $\nu=0$ and $\nu=\gamma$, a.s.,
$$ \varphi(\PP^m_{w_m} W^{(\nu)}(t))=
 \frac{\E(\psi(W^{(\nu)}_s, 0 \leq s \leq t)e^{\langle \gamma-\nu, W^{(\nu)}_t \rangle}| \F_t^\nu)}
{\E(e^{\langle \gamma-\nu, W^{(\nu)}_t\rangle}| \F_t^\nu)}.$$
\end{prop}

\begin{proof} 
By Biane et al.\ \cite{bbo2}, Theorem 5.5, we know that when $\nu=0$,  the conditional expectation is $\sigma(\PP^m_{w_m} W^{(\nu)}(t))$ measurable. So the formula defines $\varphi$ in this case. We have to prove that for this $\varphi$ the formula also holds when  $\nu=\gamma$, namely that \begin{align}\label{CM} \varphi(\PP^m_{w_m} W^{(\gamma)}(t)) =\E(\psi(W^{(\gamma)}_s,  0 \leq s \leq t)| \F_t^\gamma). \end{align}
 Let $F:{\mathcal C}_0([0,t],\R^2)\to \R$
 be a bounded measurable function. We define, for $\nu=0,\gamma$,
 \begin{align*}U_\nu&=F(\PP^m_{w_m}W^{(\nu)}_s, 0 \leq s \leq t). \end{align*}
 By Cameron--Martin--Girsanov, 
 $$\E(U_\gamma \psi(W^{(\gamma)}_s, 0 \leq s \leq t))=\E(U_0 \psi(W^{(0)}_s, 0 \leq s \leq t) e^{\langle \gamma, W^{(0)}_t \rangle})/\E(e^{\langle \gamma, W^{(0)}_t \rangle}).$$
Besides, since $U_0$ is $\F_t^0$-measurable, using the definition of $\varphi$, 
\begin{align*}\E(U_0 \psi(W^{(0)}_s, 0 \leq s \leq t) e^{\langle \gamma, W^{(0)}_t \rangle})&=
\E(U_0 \E( \psi(W^{(0)}_s, 0 \leq s \leq t) e^{\langle \gamma, W^{(0)}_t \rangle}| \F_t^0))\\
&=\E(U_0  \varphi(\PP^m_{w_m} W^{(0)}(t))\E( e^{\langle \gamma, W^{(0)}_t \rangle}| \F_t^0))\\
&=\E(U_0  \varphi(\PP^m_{w_m} W^{(0)}(t)) e^{\langle \gamma, W^{(0)}_t \rangle}).\end{align*} 
Hence, by CMG again,
$$\E(U_\gamma \psi(W^{(\gamma)}_s, 0 \leq s \leq t))=\E(U_\gamma \varphi(\PP^m_{w_m} W^{(\gamma)}(t)))$$
which gives (\ref{CM}).\end{proof}

\begin{theo}\label{expo_diedral}    Let $\mathcal{E}^m=(\mathcal{E}^m_0,\dots,\mathcal{E}^m_{m-1})$ be a random vector in $\R^m$ whose components $\mathcal{E}^m_k$  are independent exponential random variables with respective parameters $\gamma^m_{k}=\langle \gamma,  \alpha_k^m \rangle$.

(i) When $\gamma \in   \bar C_m$, for any $t>0$,  conditionally on the $\sigma$-algebra   $$\sigma(\PP^m_{w_m} W^{(\gamma)}(s), s \leq t),$$  the random vector $(\xi_0^{m}(t),\dots,\xi_{m-1}^m(t))$ is distributed as the vector $\mathcal{E}^m$ conditioned on the event  $\{\mathcal{E}^m\in \Gamma_m(\PP^m_{w_m} W^{(\gamma)}(t))\}$.

(ii) When $\gamma \in    C_m$,     the random vector $(\xi_0^{m}(\infty),\dots,\xi_{m-1}^m(\infty))$  is distributed as  $\mathcal{E}^m$ conditioned on the event  $\{\mathcal{E}^m\in \Gamma_m \}$. 
\end{theo}

\begin{proof} Let $F:\R^m\to \R$ be a  bounded and continuous function. Let us denote by $\xi^{0,m}(t)$ the string parameters of $W^{(0)}$. 
It is shown in Biane et al.\ \cite{bbo2}, Theorem 5.2, that
 the  distribution of $\xi^{0,m}(t)$ conditionally on $\F_t^{0}$ is    the normalized
Lebesgue measure 
 on $\Gamma_m(\PP^m_{w_m} W^{(0)}(t))$. Using that, from (\ref{P_eta}),  $$W_t^{(0)}=\PP^m_{w_m} W^{(0)}(t)-\sum_{k=0}^{m-1}\xi_k^{0,m}(t) \alpha^m_{k},$$  one has,  
 \begin{align*}
 \E(F(\xi^{0,m}(t))e^{\langle \gamma, W^{(0)}_t\rangle}| \F_t^0)
  &=e^{\langle \gamma, \PP^m_{w_m}W^{(0)}_t\rangle}\E(F(\xi^{0,m}(t))e^{-\langle \gamma, \sum_{k=0}^{m-1}\xi_k^{0,m}(t)\alpha_k^m\rangle}| \F_t^0)\\&=\frac{1}{c_t}
  \int {1_{\Gamma_m(\PP^m_{w_m}W^{(0)}(t))} }(x)F(x)e^{-\langle\gamma,x\rangle }\,dx   \end{align*}
where $dx$ is the Lebesgue measure on $\R^m$ and $c_t$ is the Lebesgue measure of $\Gamma_m(\PP^m_{w_m}W^{(0)}(t)))$ . If we apply Proposition \ref{prop_bayes} to the function $\psi(W^{(\gamma)}_s, 0 \leq s \leq t)=F(\xi^m(t))$ we see that
$$\E(F(\xi^m(t))| \F_t^\gamma)=\frac{\int {1_{\Gamma_m(\PP^m_{w_m}W^{(\gamma)}(t))} }(x)F(x)e^{-\langle\gamma,x\rangle }\,dx}{\int {1_{\Gamma_m(\PP^m_{w_m}W^{(\gamma)}(t))} }(x)e^{-\langle\gamma,x\rangle }\,dx},
$$ 
which proves (i). We now suppose that $\gamma \in C_m$, so \begin{align*}\E(F(\xi^m(\infty)))&=\lim_{t \to +\infty}
\E(F(\xi^m(t)))\\
&=\lim_{t \to +\infty}
\E(\E(F(\xi^m(t))| \F_t^\gamma))\\
&=\lim_{t\to +\infty} \E(\frac{\int {1_{\Gamma_m(\PP^m_{w_m}W^{(\gamma)}(t))} }(x)F(x)e^{-\langle\gamma,x\rangle }\, dx}{\int {1_{\Gamma_m(\PP^m_{w_m}W^{(\gamma)}(t))} }(x)e^{-\langle\gamma,x\rangle }\,dx})\\
&= \frac{\int {1_{\Gamma_m }(x)F(x)e^{-\langle\gamma,x\rangle }\,dx}}{\int {1_{\Gamma_m}(x) }e^{-\langle\gamma,x\rangle }\,dx},\end{align*}
since $\PP^m_{w_m}W^{(\gamma)}(t)/t$ converges to $\gamma$ as $t$ tends to $+\infty$ and thus $\Gamma_m(\PP^m_{w_m}W^{(\gamma)}(t))$ tends to $\Gamma_m$ for $\gamma\in C_m$, a.s.. This proves (ii). \end{proof}

A direct consequence of the theorem is the following:
\begin{coro}\label{condxi}
The conditional distribution of $\xi^m(t)$ given  $\sigma(\PP^m_{w_m} W^{(\gamma)}(s), s \leq t)$ is the one of  $\xi^m(\infty)$
given $\{\xi^m(\infty)\in \Gamma_m(\PP^m_{w_m} W^{(\gamma)}(t))\}$. 
\end{coro}

\begin{coro} \label{Verma}We suppose that $\gamma \in C_m$. Let $\xi^m(\infty)$ be the  Verma string parameters of $W^{(\gamma)}$.  If we let 
  $$(\xi_0^{m}(\infty),\frac{\xi^{m}_{1}(\infty)}{a^m_{1}}-\frac{\xi_2^{m}(\infty)}{a^m_2},\frac{\xi^{m}_{2}(\infty)}{a^m_{2}}-\frac{\xi_3^{m}(\infty)}{a^m_3},\cdots,\frac{\xi_{m-1}^{m}(\infty)}{a^m_{m-1}})$$
equals to
$$(\frac{\varepsilon^m_0}{\gamma^m_0},\frac{\varepsilon^m_{1}}{\gamma^m_1a^m_1},\frac{\varepsilon^m_{2}}{\gamma^m_1a^m_1+\gamma^m_2a^m_2},\cdots,\frac{\varepsilon^m_{m-1}}{\gamma^m_1a^m_1+\cdots+\gamma^m_{m-1}a_{m-1}^m}),$$
then  the $\varepsilon^m_n, n=0,1, \cdots, m-1$   are independent exponential random variables with parameter 1. 
\end{coro}
\begin{proof} This follows from 
Theorem \ref{expo_diedral} since, on $\R_+^m$,
\begin{align*}
\P(\xi_0^m(\infty)\in & dx_0,\cdots,\xi_{m-1}^m(\infty)\in dx_{m-1})\\ &=Ce^{-\gamma^m_0x_0}\, dx_0\prod_{k=2}^{m}e^{-\gamma^m_{k-1}x_{k-1}}1_{\{\frac{x_{k-1}}{a^m_{k-1}}\ge\frac{x_{k}}{a^m_{k}}\}}\, dx_{k-1}\\&=Ce^{-\gamma^m_0 x_0}dx_0\prod_{k=2}^{m}e^{-(\gamma^m_1a^m_1+\cdots+\gamma^m_{k-1}a^m_{k-1})(\frac{x_{k-1}}{a^m_{k-1}}  - \frac{x_{k}}{a^m_{k}})}1_{\{\frac{x_{k-1}}{a^m_{k-1}}\ge\frac{x_{k}}{a^m_{k}}\}}\, dx_{k-1}
\end{align*}
where by convention $x_m/a_m^m=0$ and $C$ is a normalizing constant.
\end{proof}
Notice that this is similar to Renyi's representation of order statistics \cite{renyi}.   
\section{Affine string parameters of the space-time Brownian motion}\label{DihtoAff}

\subsection{Pitman and L\'evy transforms}

The infinite dihedral group $I(\infty)$ is the infinite group generated by two involutions $s_0, s_1$ with no relation.    In   $V=\R^2$ identified with its dual $\tilde V$, let $(\alpha_0,\tilde \alpha_0), (\alpha_1,\tilde \alpha_1)$  in $ V  \times \tilde V$ be given by
  \begin{align*}
\left\{
    \begin{array}{l}
      \alpha_0=(0, -2), \\
        \alpha_1=(0,2), \\
\tilde \alpha_0= (1, -1), \\
\tilde \alpha_1 =(0,1).\\
    \end{array}
\right.
\end{align*} 
 The matrix $( \tilde \alpha_i(\alpha_j))_{0\leq i ,j \leq 1}$ is the Cartan matrix
 $$
\begin{pmatrix} 
2 & -2  \\
-2 & 2  \\
\end{pmatrix}. $$ 
The two linear reflections $s_0,s_1$ of $\R^2$ defined by  $$s_i(v)=v-\tilde \alpha_i(v)\alpha_i,$$
 for $v \in \R^2$, generate the group $I(\infty)$. Notice that
 $s_0$ is a non orthogonal reflection. For each $k\in \N$,  when we write $\alpha_k, \tilde \alpha_k, s_k,\cdots, $ we take $k$ modulo 2, as above.  Thus
  $\alpha_k=(-1)^k \alpha_0$ (but $\tilde \alpha_k\neq(-1)^k \tilde \alpha_0$). Notice that 
 $$C_{\mbox{aff}}=\{(t,x)\in \R^2; 0 < x < t\}=\{v \in V; \tilde \alpha_i(v) >0, i=0,1\}.$$

 We define
the Pitman
    transforms $\PP^{}_{s_0}$ and $\PP^{}_{s_1}$    on the paths of ${\mathcal C}_0(\R^2)$  by 
\begin{align*}
\PP_{s_i} \eta(t)&=\eta(t)-\inf_{0\le  s\le t}\tilde \alpha_i(
\eta(s))\alpha_i,
\end{align*}
where $\eta\in{\mathcal C}_0(\R^2), i=0,1$. Let $T:\R^2\to \R^2$ be the involutive transformation $T(t,x)=(t,t-x)$. Then
 \begin{align}\label{entrela}\mathcal P_{s_1}T=T\mathcal P_{s_0}.
\end{align}
 We will use mainly space-time paths, i.e.\  paths which can be written as $\eta(t)=(t,\varphi(t))$,  $t\ge 0$, where $\varphi(t)\in \R$. In this case,
\begin{align*}\mathcal P_{s_0}\eta(t)&=(t,\varphi(t)+2\inf_{s\le t}(s-\varphi(s))\\
\mathcal P_{s_1}\eta(t)&=(t,\varphi(t)-2\inf_{s\le t}\varphi(s)).
\end{align*}
One  recognizes in the second component the transforms $\mathcal P_0, \mathcal P_1$ defined in Section 1.2.

\subsection{Affine string parameters}
We consider the space-time process
$B_t^{(\mu)}=(t,B_t+t \mu)$ where $B$ is a standard real Brownian motion starting from $0$.   We define, for $t \leq \infty$, $\xi(t)=\{\xi_k(t),k\geq 0\}$  by
 $$\xi_k(t)=-
    \inf_{0\leq s\leq t}\tilde \alpha_k( \PP_{s_{{k-1}}}\ldots \PP_{s_{1}} \PP_{s_{0}}B^{(\mu)}(s)).$$ Hence $$\xi_0(t)=-
    \inf_{0\leq s\leq t}\tilde \alpha_0( B^{(\mu)}(s)).$$ 
    
      \begin{defi}We call $\xi(t)=\{\xi_k(t), k \geq0\}$  the affine string parameters of $ B^{(\mu)}$ on $[0,t]$, and $\xi(\infty)=\{\xi_k(\infty), k \geq 0\}$   its  Verma affine string parameters.\end{defi}
    We use a terminology inspired by the Kac-Moody affine algebra $A_1^{(1)}$ (see in particular Kac \cite{Kac}, Kashiwara \cite{kash93}, and Section \ref{affine} below).
One has, for any $n \geq 0$,
\begin{align}\label{formsum}\PP_{s_n}\cdots \PP_{s_0}B^{(\mu)}(t)=B^{(\mu)}_t+\sum_{k=0}^n \xi_k(t)\alpha_k. \end{align}
When $\mu\ne 0,1$, 
by the law of large numbers, a.s., $\lim_{t \to +\infty} B_t^{(\mu)}/t =(1,\mu) \in C_{\mbox{aff}}$, hence 
$$\lim_{t \to +\infty} \tilde\alpha_0(B_t^{(\mu)})=\lim_{t \to +\infty}\tilde\alpha_1(B_t^{(\mu)})=+\infty,$$
and by the analogue of Proposition \ref{x_fini}, $\xi_k(\infty)<+\infty$ for each $k\geq0$.

Recall that, for $v=(t,x)\in \R^2$,  $$\tau_m v=(\frac{\pi t}{m},x).$$
We will use  that for $v\in \R^2$,
\begin{align}\label{tauv}\tau_mv=(\frac{\pi}{m\sin{\frac{\pi}{m}}}(\tilde \alpha_0^m(v)+\tilde \alpha_1^m(v)\cos{\frac{\pi}{m}}), \tilde \alpha_1^m(v)),\end{align}
so the asymptotics of $\tau_mv$ and $(\tilde \alpha_0^m(v)+\tilde \alpha_1^m(v),\tilde \alpha_1^m(v))$ are the same as $m$ tends to $+\infty$. It is easy to see that,
\begin{lemm}\label{lemm_tau}
For $i=0,1,$
\begin{align*}\lim_{m \to +\infty}&\tau_m\alpha_i^m=\alpha_i,\\ \lim_{m \to +\infty}&\tau_m^{-1}\tilde \alpha_i^m=\tilde \alpha_i,\\ \lim_{m \to +\infty}&\tau_m \circ s^{m}_{i} \circ 
\tau_m^{-1}= s_{i}.\end{align*}
\end{lemm}

\begin{defi}\label{lambdam} The dihedral  highest weight process is
$$\Lambda^{(\mu)}_m(t)= \PP^m_{w_m}W^{(\frac{m}{\pi},\mu)}(t), \ t \geq 0.$$
\end{defi}
By Theorem \ref{thmconvTau}, $\tau_mW^{(\frac{m}{\pi},\mu)}$ tends in distribution to $B^{(\mu)}$ and $\tau_m \Lambda^{(\mu)}_m$ tends to $A^{(\mu)}$. When $\mu\ne 0,1$, we will always suppose that $m$ is large enough to have $(\frac{m}{\pi},\mu)\in C_m$. This ensures that each $\xi^m_k(\infty)$ is finite.

\begin{prop} \label{conv_ps} 
Almost surely, for all  $t \geq 0$ and $k \in \N$,
$$\lim_{m\to +\infty} \xi_k^{m}(t)=\xi_k(t),$$
and when $\mu\ne 0,1$,
$$\lim_{m\to +\infty} \xi^{m}_k(\infty)=\xi_k(\infty).$$ 
\end{prop}

\begin{proof}
Let $	\eta^m_0=W^{(\frac{m}{\pi},\mu)}$, $\eta_0=B^{(\mu)}$ and, for $k\geq 1$,
$$\eta^m_{k}=\PP^m_{s^m_{{k-1}}}\ldots \PP^m_{s^m_{0}}\eta^m_0, \, \eta_{k}=\PP_{s_{{k-1}}}\ldots \PP_{s_{0}}\eta_0.$$
We prove by 
induction on $k \in \N$, that for any $T>0$, a.s., uniformly on $[0,T]$, 
\begin{align*}
\lim_{m \to +\infty}\tau_m\eta^m_{k}(t) =\eta_{k}(t) \\
 \lim_{m \to +\infty} \xi_k^{m}(t)=\xi_k(t).
\end{align*}
This is clear for $k=0$. If it holds for $k-1$, then, by Lemma \ref{lemm_tau},
$$\tau_m\eta^m_{k}(t) =\tau_m\PP^m_{s^m_{k-1}}\eta^m_{k-1}(t)=\tau_m\eta^m_{k-1}(t)+\xi_{k-1}^m(t) \tau_m \alpha_{k-1}^m$$
 converges uniformly to  $\eta_{k}(t) $. Since one can write  
 $$\xi_{k}^{m}(t)=-\inf_{0\leq s \leq t}\tilde \alpha_k^m( \eta_k^m(s))=-\inf_{0\leq s\leq t}\tau_m^{-1}(\tilde \alpha_{k}^{m})(\tau_m\eta^m_{k}(s)),$$ $\xi_{k}^{m}(t)$ converges to $\xi_{k}(t)$.

We now suppose that $\mu\ne 0,1$. First we notice that, a.s.,  for any $C>0$, there is a time $t_C >0 $ such that, for any $m\in \N$, 
\begin{align}\label{enough}\tilde \alpha_i^m( \eta_0^m(t) )\geq C,\end{align} 
when $t \geq t_C$, $i=0,1$. Let us show that a.s., for $t$ large enough (independent of $m$), $\xi_k^{m}(\infty)= \xi_k^{m}(t)$ for all $m\in \N$, large enough. This will imply that  $\xi_k^{m}(\infty)$ tends to $\xi_k(\infty)$.  Since
 $\tilde \alpha_k(\eta_k^m(0))=0$, it is enough to show that almost surely, there exists $T_k > 0$ such that for $t \geq T_k$, for all $m \in \N$,
$\tilde \alpha_k^m( \eta_k^m(t)) \geq 0.$
This follows from (\ref{enough}) and from the fact, easily proved by induction on $k$, that almost surely, there is a $C_k>0$ such that, for all $t \geq 0, m\in \N$,
$\|\eta_k^m(t)-\eta_0^m(t)\| \leq C_k.$
\end{proof}

 \begin{theo} \label{loidexi} For $\mu\ne 0,1$,
 one can write $\xi_0(\infty)=\varepsilon_0/2(1- \mu),$ and for all $k \geq 1$,
$$\frac{\xi_k(\infty)}{k}=\sum_{n=k}^{+\infty} \frac{ 2\varepsilon_n}{n(n+1)+(1-2\mu)\nu_n},$$
where $\varepsilon_n, n \geq0$, are independent exponential random variables with parameter $1$, and
where $\nu_n=n$ when $n$ is even and $\nu_n=-(n+1)$ when $n$ is odd.
\end{theo}
\begin{proof}   
We consider the independent exponentially distributed random variables $\varepsilon^m_n$  of  Corollary \ref{Verma} for the drift $\gamma=(m/\pi,\mu)$. We have
$\varepsilon_0^m=\gamma_0^m \xi_0^m(\infty)$ and for $1 \leq k < m$,
$$\varepsilon_k^m=(\sum_{n=1}^k\gamma_n^m a_n^m)(\frac{\xi^{m}_{k}(\infty)}{a^m_{k}}-\frac{\xi_{k+1}^{m}(\infty)}{a^m_{k+1}}).$$
One has 
\begin{align*}
\gamma_{2k}^m&=\langle (\frac{m}{\pi},\mu), \alpha_0^m\rangle=2(\frac{m}{\pi}\sin(\frac{\pi}{m})-\mu\cos(\frac{\pi}{m})), \\\gamma_{2k+1}^m&=\langle (\frac{m}{\pi},\mu), \alpha_1^m \rangle=2\mu.\end{align*}
When $m \to +\infty$,  $\gamma_{2k}^m\to 2(1-\mu)$ and $a_k^m/m \to k\pi$ hence
$$\lim_{m\to +\infty}m\sum_{n=1}^k\gamma_n^m a_n^m =\frac{1}{2}({n(n+1)+(1-2\mu)\nu_n}).$$
So it follows from Proposition \ref{conv_ps} that the variables $\varepsilon_n$ defined by, a.s., for $n \geq 1$,
\begin{align}\label{epsi}\varepsilon_n:=\lim_{m \to +\infty}\varepsilon_k^m=\frac{{n(n+1)+(1-2\mu)\nu_n}}{2}(\frac{\xi_n(\infty)}{n}- \frac{\xi_{n+1}(\infty)}{n+1})\end{align}
and ${\varepsilon_0}={2(1- \mu)}\xi_0(\infty)$ are
 independent exponentially distributed random variables with parameter 1.  By definition of the cone $\Gamma_m$, the sequence $\xi_k^m(\infty)/a_k^m,$ $ k \geq 1,$ is non negative and decreasing. Therefore 
  $\xi_k(\infty)/k, k \geq 1,$ is also non negative and decreasing. We denote by $S$ its limit, which is a non negative random variable. Thus one has, by (\ref{epsi}),
$$\frac{\xi_k(\infty)}{k}=\sum_{n=k}^{+\infty} \frac{ 2\varepsilon_n}{n(n+1)+(1-2\mu)\nu_n}+S.$$
Moreover Fatou's lemma implies that
\begin{align}\label{Fatou} \E(\frac{\xi_k(\infty)}{k})\le \liminf_m \E(\frac{\xi_k^m(\infty)}{k})=\liminf_m \frac{\pi}{m}\sum_{n=k}^{m-1}{\frac{1}{a_1^m\gamma^m_1+\cdots+a_{n}^m\gamma^m_{n}}}.
\end{align}
 Let $c=\max\{1/\gamma_i^m, m\in \N, i=0,1\}$. Since
$0 \leq {2t} \leq \sin \pi t  $
when $0 \leq t \leq 1/2$, one has, 
for $k \leq [m/2]$, $$\sum_{n=1}^k a_n^m\gamma_n^m  \geq \frac{1}{c}\sum_{n=1}^k \sin(\frac{n \pi}{m}) \geq \sum_{n=1}^k \frac{2n}{cm}\geq \frac{k^2}{cm},$$
and, for $k > [m/2]$, $$\sum_{n=1}^{k} a_n^m\gamma_n^m  \geq \sum_{n=1}^{[m/2]+1} a_n^m\gamma_n^m  \geq \frac{m}{4c}.$$
Therefore, for $N\le m-1$,
\begin{align*}
\frac{\pi}{m}\sum_{n=N}^{m-1}{\frac{1}{a_1^m\gamma^m_1+\cdots+a_{n}^m\gamma^m_{n}}}&\leq {c\pi}(\sum_{n=N}^{[m/2]}\frac{1}{n^2}+\sum_{n=[m/2]+1}^m \frac{4}{m^2})
\end{align*}
is  as small as we want for $N$ large enough. As  one has for every fixed $N>0,$
$$\lim_m \frac{\pi}{m}\sum_{n=k}^{N}{\frac{1}{a_1^m\gamma^m_1+\cdots+a_{n}^m\gamma^m_{n}}}=\sum_{n=k}^{N} \frac{ 2}{n(n+1)+(1-2\mu)\nu_n},$$  the $\liminf $ in (\ref{Fatou}) 
equals  $\sum_{n=k}^{+\infty} \frac{ 2}{n(n+1)+(1-2\mu)\nu_n}$ and  therefore
$$\E(\frac{\xi_k(\infty)}{k})\le \sum_{n=k}^{+\infty} \frac{ 2}{n(n+1)+(1-2\mu)\nu_n}.$$
This proves that $\E(S)=0$, which implies that $S=0$ a.s.\  and finishes the proof.
\end{proof}

An important result for us is the following ( it is maybe unexpected because the string coordinates $x_k(t)$ for $C^1$ paths vanish for $k$ large enough, see Proposition \ref{nombrefini}).

\begin{theo}\label{lim2} For $\mu\ne 0,1$, almost surely,
$\lim_{k \to +\infty} \xi_k(\infty)= 2.$
\end{theo}
\begin{proof}
Since
$$\E({\xi_k(\infty)})=\sum_{n=k}^{+\infty} \frac{ 2k}{n(n+1)+(1-2\mu)\nu_n},$$
there is a $C>0$ such that
\begin{align*}|\E(\xi_k(\infty))-\sum_{n=k}^{+\infty} \frac{ 2k}{n(n+1)}|\leq C k\sum_{n=k}^{+\infty} \frac{1}{n^3}.\end{align*}
Since $\sum_{n=k}^{+\infty} \frac{ k}{n(n+1)}=1$, 
 $\E(\xi_k(\infty))$ tends to $2$, as $k \to +\infty$. To prove that  $\xi_k(\infty)$ tends to $2$ almost surely it suffices to  remark that 
if $X_n$ is a sequence of i.i.d. centered random variables with a finite fourth moment then, almost surely,
$Z_k=k\sum_{n=k}^{+\infty}  {X_n}/{n^2}$ tends to $0$.
Indeed, it is easy to see that 
$$\E(Z_k^4)=k^4 \E(X_1^4)\sum_{k \leq n} \frac{1}{n^8}+ 2k^4\E(X_1^2)^2\sum_{k \leq n} \frac{1}{n^4} \sum_{n<m}\frac{1}{m^4}.$$
For $p>1$,
$\sum_{n > k} {n^{-p}} \leq C_p {k^{1-p}}$
for some $C_p>0$, hence for some $C>0$
$$\E(Z_k^4)\leq C k^4(\frac{1}{k^7}+  \frac{1}{k^6} ).$$
This shows that $\E(\sum_{k} Z_k^4) < +\infty$ which implies that $Z_k \to 0$, a.s.
\end{proof}

\subsection{Affine Verma weight $L^{(\mu)}(\infty)$}

 We consider, for $\mu\ne 0,1,$  and  $k\ge 1,$  \begin{align*}
M_k(\infty)=\frac{1}{2}\xi_{k}(\infty)\alpha_k+\sum_{n=0}^{k-1} \xi_n(\infty)\alpha_n,
\end{align*}
 where $\xi(\infty)=\{\xi_k(\infty), k \geq 0\}$ are the affine Verma string parameters  of $ B^{(\mu)}$. Notice that, more simply,  since $\alpha_n=(0,(-1)^{n+1}2)$, $$M_k(\infty)=(0,(-1)^{k+1} \xi_k(\infty)+2\sum_{n=0}^{k-1}(-1)^{n+1}\xi_n(\infty)).$$
The notation with the $\alpha_n$'s may be strange for the reader; it is explained by its natural interpretation in $A_1^{(1)}$ (see \ref{secKac}).

\begin{prop} \label{limite_L}  Let $\mu\ne 0,1$. When $k$ goes to infinity, $M_k(\infty)$ converges almost surely and in $L^2$ to 
 $$L^{(\mu)}(\infty)=\frac{1}{2}\sum_{n=0}^{\infty}(\frac{\varepsilon_{2n}}{n+1-\mu}\alpha_0+\frac{\varepsilon_{2n+1}}{n+\mu}\alpha_1).$$\end{prop}

\begin{proof} One has the key relation
\begin{align*}M_{2p+2}(\infty)-M_{2p}(\infty)&=(0, -\xi_{2p}(\infty)+2\xi_{2p+1}(\infty)-\xi_{2p+2}(\infty))\\
&=(0,\frac{\varepsilon_{2p+1}}{p+\mu}-\frac{\varepsilon_{2p}}{p+1-\mu}).
\end{align*}
Hence $\{M_{2p}(\infty)-\E(M_{2p}(\infty)),p\ge 0\}$ is a martingale 
bounded in $L_2$. Since  $\E(M_{2p}(\infty))$ converges, this shows that the sequence $\{M_{2p}(\infty), p\in \N\}$ converges a.s.\ and in $L^2$ and that  its limit is $L^{(\mu)}(\infty)$. As $$M_{2p+1}(\infty)=M_{2p}(\infty)+\frac{1}{2}(\xi_{2p}(\infty)-\xi_{2p+1}(\infty))\alpha_0$$ and $\xi_k(\infty)$ tends to $2$ by Theorem \ref{lim2}, the sequence $  \{M_{2p+1}(\infty),p\in \N\}$ has the same limit, namely  $L^{(\mu)}(\infty)$. \end{proof}
 Notice that $L^{(\mu)}(\infty)$ is not the limit of the series  $\sum_{k=0}^n\xi_k(\infty)\alpha_k$ since $\xi_k (\infty)$ tends to $2$ (Theorem \ref{lim2}). 
\begin{defi} We call 
$L^{(\mu)}(\infty)$ the affine Verma weight of $B^{(\mu)}$.
\end{defi}
The terminology is due to the fact, as we will see in Section \ref{affine}, that $-L^{(\mu)}(\infty)$ plays the role of an asymptotic weight in the Verma module of highest weight $0$ in the case of the affine Lie algebra $A_1^{(1)}$, modulo $\delta$.
In coordinates, $L^{(\mu)}(\infty)$ can be written as $(0,D^{\mu}(\infty))$ where   \begin{align}\label{Dmu}D^{\mu}(\infty)=\sum_{n=0}^{\infty}(\frac{\varepsilon_{2n+1}}{n+\mu}-\frac{\varepsilon_{2n}}{n+1-\mu}).\end{align} 
 For a sequence $x=(x_k)\in \R_+^{\N}$, we let
\begin{align}\label{sigma}\omega(x)=\lim_{n \to +\infty} \sum_{k=0}^{n-1}x_k\alpha_k+\frac{1}{2}x_n\alpha_n,\end{align}
and write $\omega(x)\in \R^2$ when this limit exists in $\R^2$. When $\mu\ne 0,1$, 
$$\omega(\xi(\infty))= L^{(\mu)}(\infty). $$
  \begin{defi} One defines,  for $\lambda\in \bar{ C}_{\mbox{aff}} $, 
\begin{align*}&\Gamma=\{x=(x_k)\in  \R^\N : \frac{x_k}{k}\ge \frac{x_{k+1}}{k+1}\ge 0,  \textrm{ for all } k\ge 1,\,  x_0\ge 0, \,  \omega(x) \in\R^2\},\\
&\Gamma(\lambda)=\{x\in \Gamma : x_k\le \tilde \alpha_k(\lambda-\omega(x)+\sum_{i=0}^kx_i\alpha_i), \textrm{ for every }  k\ge 0\}.
\end{align*}
\end{defi}
We remark that $x\in \Gamma$ is in $\Gamma(\lambda)$ if and only if for all $k\geq 0$,\begin{align}\label{condG} \tilde \alpha_k(\omega(x)-\sum_{i=0}^{k-1}x_i\alpha_i-\frac{1}{2}x_k\alpha_k)\leq \tilde \alpha_k(\lambda) .\end{align}
Notice the occurence of the coefficient $1/2$ here which will explain the correction term in the representation theorem. The same coefficient already occurs in the description of the crystal $B(\lambda)$ of $A_1^{(1)}$ given by Nakashima  \cite{nakash2} which is the discrete analogue of $\Gamma(\lambda)$.  In a sense $\Gamma$ will have the role of a continuous Verma crystal and  $\omega(x)$ has the role of the weight of  $x\in \Gamma$. In the same way $\Gamma(\lambda)$ will be a kind of continuous crystal of highest weight $\lambda$.

\begin{prop}\label{positif}
When $\mu\ne 0,1$, the Verma string parameters 
$\xi(\infty)$ of $B^{(\mu)}$ are a.s.\ in $\cup\{\Gamma(\lambda), {\lambda \in C_{\mbox{aff}}}\}$ 
 and for each $\lambda \in C_{\mbox{aff}}$,
$\P(\xi(\infty)\in \Gamma(\lambda))>0.$ 
\end{prop}
\begin{proof} Since $\omega(\xi(\infty))=L^{(\mu)}(\infty)$ is the limit of $M_k(\infty)$,  $\xi(\infty)\in \Gamma(\lambda)$ for some $\lambda$ large enough. Now
 \begin{align*}L^{(\mu)}(\infty)-M_{2k}(\infty)=\sum_{p=k}^{\infty}(M_{2p+2}(\infty)-M_{2p}(\infty))=(0,R_k)
\end{align*}
 where $$R_k=\sum_{n= k}^\infty(\frac{\varepsilon_{2n+1}}{n+\mu}-\frac{\varepsilon_{2n}}{n+1-\mu}).$$ 
Let  $T_k=\frac{\varepsilon_{2k+1}}{k+\mu}$ and $$S_k=R_{k+1}-\sum_{n=2k+1}^{+\infty} \frac{ 2\varepsilon_n}{n(n+1)+(1-2\mu)\nu_n}.$$ 
Then
\begin{align*}L^{(\mu)}(\infty)-M_{2k+1}(\infty)&=L^{(\mu)}(\infty)-M_{2k}(\infty)-(M_{2k+1}(\infty)-M_{2k}(\infty))\\ &= (0,R_k-(\xi_{2k+1}(\infty)-\xi_{2k}(\infty)))\\&=(0,S_k+T_k)\end{align*}
since 
$$\xi_{k+1}(\infty)-\xi_k(\infty)=\frac{\xi_{k+1}(\infty)}{k+1}-\frac{2k\varepsilon_k}{k(k+1)+(1-2\mu)\nu_k}.$$ 
Let $a<0<b$. 
 As the sequences $R_n,S_n,T_n$ converge almost surely to $0$, we can choose $n_0\ge 1$ such that  $$\P(R_n, S_n,T_n\in[a/2,b/2]\mbox{ for all } n \geq n_0)>0.$$
   One has 
\begin{align*}
\P(&R_n, S_n+T_n\in[a,b],\mbox{for all }n \geq 0)
  \\
& \ge \P(R_{k}-R_{n_0},T_k+S_{k}-S_{n_0}\in[a/2,b/2], \mbox{for } 0 \leq k <n_0, \\ 
&\quad \quad \quad R_n,S_n,T_n\in[a/2,b/2],  \mbox{for } n \geq n_0) \\ & \geq  \P(R_{k}-R_{n_0} , T_k+S_{k}-S_{n_0}\in[a/2,b/2], \mbox{for } 0 \leq k <n_0)\\ 
&\quad   \quad\times \P( R_n, S_n,T_n\in[a/2,b/2], \mbox{for } n \geq n_0) >0.
\end{align*}
This allows to finish the proof, since, by (\ref{condG}), $\xi(\infty)\in \Gamma(\lambda)$ if and only if, for all $k  \geq 0$, 
\begin{align}\label{majalpha}\tilde{\alpha}_k(L^{(\mu)}(\infty)-M_k(\infty))\leq  \tilde{\alpha}_k(\lambda)\end{align} where for a given $\lambda \in C_{\mbox{aff}}$, $\tilde{\alpha}_k(\lambda)$ takes only two positive values.
\end{proof}

In order to prove Proposition 5.14 below, and to study the string parameters $\xi_k(t)$ in the next section we need to show some uniform approximations by the dihedral case. 
We take here $\mu\ne 0,1,$  and $m$ large enough to ensure that $(\frac{m}{\pi},\mu)\in C_m$. Let$$M_k^m(\infty)=\frac{1}{2}\xi_{k}^m(\infty)\alpha_{k}^m+\sum_{n=0}^{k-1}\xi_n^m(\infty)\alpha_n^m,$$
 for $0 \leq k \leq m$, with the convention that $\xi^m_m(\infty)=0$.

\begin{prop}  \label{limite_m} When $\mu\ne 0,1,$ for any $k\ge 0$, $\tau_mM_k^m(\infty)$ converges to $M_k(\infty)$  a.s.\ when $m$ goes to infinity.  
\end{prop}

\begin{proof} This follows  from Lemma \ref{lemm_tau} and Proposition \ref{conv_ps}.\end{proof}

\begin{prop} \label{convunifM}  When $\mu\ne 0,1,$ one has, in probability, \begin{align}\lim_{m\to +\infty}\sup_{1 \leq k \leq m}  \| \tau_m M^m_k(\infty)- M_k(\infty)\|=0.\label{lim1}\end{align}
\end{prop} 

\begin{proof} 
We use Corollary \ref{Verma} and its notations. 
For  $1\le 2p+1 \leq m,$
$$M_{2p+2}^m(\infty)-M_{2p}^m(\infty)=\frac{1}{2}\xi^m_{2p+2}(\infty)\alpha_{2p+2}^m+\xi^m_{2p+1}(\infty)\alpha_{2p+1}^m+\frac{1}{2}\xi^m_{2p}(\infty)\alpha_{2p}^m.$$
Therefore,  
\begin{align*}\tilde \alpha_0^m(M_{2p+2}^m(\infty)-M_{2p}^m(\infty))&=\xi^m_{2p+2}(\infty)-2\xi^m_{2p+1}(\infty)\cos\frac{\pi}{m}+\xi^m_{2p}(\infty)\\&=-\frac{2\varepsilon^m_{2p+1}a_{2p+1}^m\cos\frac{\pi}{m}}{\gamma^m_1a_1^m+\gamma^m_0a_2^m+\cdots+\gamma^m_1a_{2p+1}^m}\\&    +a_{2p}^m(\frac{\varepsilon^m_{2p+1}}{\gamma^m_1a_1^m+\cdots+\gamma^m_1a_{2p+1}^m}+\frac{\varepsilon^m_{2p}}{\gamma^m_1a_1^m+\cdots+\gamma^m_0a_{2p}^m}),\end{align*}
since
 $$a^{m}_{2p+2}-2a^{m}_{2p+1}\cos\frac{\pi}{m}+a^{m}_{2p}=0.$$
 One deduces that
\begin{align*} 
\tilde \alpha_0^m(M_{2p}^m(\infty))&=2\sin\frac{\pi}{m}\sum_{l=0}^{p-1}(\frac{\varepsilon^m_{2l}\cos\frac{l\pi}{m}}{\gamma^m_1\sin\frac{l}{m}\pi+\gamma^m_0\sin \frac{l+1}{m}\pi}-\frac{\varepsilon^m_{2l+1}\cos\frac{(l+1)\pi}{m}}{\gamma^m_1\sin\frac{l+1}{m}\pi+\gamma^m_0\sin\frac{l}{m}\pi})
\end{align*}
by using the relations
\begin{align*}
\sum_{k=1}^{n}a^m_{2k}=\frac{\sin\frac{n\pi}{m}\sin\frac{(n+1)\pi}{m}}{\sin\frac{\pi}{m}},
\sum_{k=1}^{n}a^m_{2k-1}=\frac{\sin^2\frac{n\pi}{m}}{\sin\frac{\pi}{m}},
\sum_{k=1}^{n}a^m_k=\frac{\sin\frac{n\pi}{2m}\sin\frac{(n+1)\pi}{2m}}{\sin\frac{\pi}{2m}}.
\end{align*}
Similarly, one finds
\begin{align*}\tilde \alpha^m_1(M^m_{2p}(\infty))&=-2\xi_0^m(\infty)\cos\frac{\pi}{m}+\xi_{2p}^m(\infty)\sin\frac{\pi}{m}\tan \frac{p\pi}{m}\\+&2\sin\frac{\pi}{m}\sum_{l=1}^{p-1}(\frac{\varepsilon^m_{2l+1}\cos\frac{l\pi}{m}}{\gamma^m_1\sin\frac{l+1}{m}\pi+\gamma_0^m\sin \frac{l}{m}\pi}-\frac{\varepsilon^m_{2l}\cos \frac{(l+1)\pi}{m}}{\gamma_1^m\sin\frac{l}{m}\pi+\gamma_0^m\sin \frac{l+1}{m}\pi}).
\end{align*}
On the other hand, one has
$$M_{2p+1}^m(\infty)-M_{2p}^m(\infty)=\frac{1}{2}(\xi_{2p}^m(\infty)\alpha_0^m+\xi_{2p+1}^m(\infty)\alpha_1^m).$$
Thus the proposition follows from the next lemma, from (\ref{epsi}) and from Lemma \ref{lemm_tau}.
\end{proof}

\begin{lemm} \label{convunifcor}  Let $\gamma_0^m,\gamma_1^m, m \in \N,$ be real numbers such that $\gamma_0^m$ tends to $ 2(1-\mu)$ and $\gamma_1^m$ tends to  $2\mu$ when $m$ tends  to $+\infty$. Let, for $p\in\{0,\cdots,[m/2]\},$
\begin{align*}
 S^m_p&=\sin\frac{\pi}{m}\sum_{l=0}^{p-1}(\frac{\varepsilon^m_{2l}\cos\frac{l\pi}{m}}{\gamma_1^m\sin\frac{l}{m}\pi+\gamma_0^m\sin\frac{l+1}{m}\pi}-\frac{\varepsilon^m_{2l+1}\cos\frac{(l+1)\pi}{m}}{\gamma_1^m\sin\frac{l+1}{m}\pi+\gamma_0^m\sin\frac{l}{m}\pi}),\\
 S_p&=\sum_{l=0}^{p-1}(\frac{\varepsilon_{2l}}{2  l+2(1-\mu) }-\frac{\varepsilon_{2l+1}}{2l+2\mu}),
\end{align*}
where $(\varepsilon_n^m)_{n\ge 0}$ is a sequence of independent exponentially distributed random variables with parameter $1$ that   converges almost surely to $(\varepsilon_n)_{n\ge 0}$ when $m$ goes to infinity.  
Then  in probability
 $$\lim_{m\to \infty} \sup_{1 \leq p \leq m/2}\vert  S_p^m- S_p\vert=0.$$
\end{lemm}

\begin{proof} We have 
\begin{align}\label{first}
 \sup_{1 \leq p \leq m/2}\vert  S_p^m- S_p\vert \le \sup_{1 \leq p \leq m/2}\vert  S_p^m- S_p- \E(S_p^m- S_p)\vert+ \sup_{1 \leq p \leq m/2}\vert  \E(S_p^m- S_p)\vert.
\end{align}
Let us show that the first term of the right-hand side converges to $0$ in probability.  One has for $N\le m/2$, 
\begin{align} \label{second}
 \sup_{1 \leq p \leq m/2}\vert  S_p^m- S_p- \E(S_p^m- S_p)\vert & \le 2 \sup_{1 \leq p \leq N}\vert  S_p^m- S_p- \E(S_p^m- S_p)\vert 
\\ & \quad          +  \sup_{N \leq p \leq m/2}\vert  \tilde S_p^m- \tilde S_p- \E(\tilde S_p^m- \tilde S_p)\vert,  \nonumber
\end{align}
where $\tilde{S}_p^m=S_p^m-S_N^m$ and $\tilde{S}_p=S_p-S_N$, for $p\ge N$.
The first term of the right hand-side of (\ref{second}) converges to $0$, for any $N$, when $m$ goes to infinity. As  $\{\tilde S_{N+k}^m-\tilde S_{N+k}-\E(\tilde S_{N+k}^m-\tilde S_{N+k}),k=0,\cdots,[m/2]-N\}$ is a martingale, Doob's martingale inequality gives, for $a>0$,
\begin{align}\label{third}
\P(\sup_{N \leq p \leq m/2}\vert  \tilde S_p^m-\tilde S_p- \E(\tilde S_p^m- \tilde S_p)\vert\ge {a})\le \frac{1}{a^2}{\mbox Var}(\tilde S^m_{[m/2]}-\tilde S_{[m/2]}).
\end{align}
Besides, one has 
$${\mbox Var}(\tilde S^m_{[m/2]}-\tilde S_{[m/2]})
\le 2{\mbox Var}(\tilde S^m_{[m/2]})+2{\mbox Var}(\tilde S_{[m/2]}). $$  
One has
\begin{align*} 
{\mbox Var}(\tilde S^m_{[m/2]})
&\le \sum_{l=N+1}^{[m/2]-1}\frac{\sin^2\pi/m}{(\gamma_1^m\sin\frac{l}{m}\pi+\gamma_0^m\sin\frac{l+1}{m}\pi)^2}+\frac{\sin^2\pi/m}{(\gamma_1^m\sin\frac{l+1}{m}\pi+\gamma_0^m\sin\frac{l}{m}\pi)^2},
\end{align*} 
and 
\begin{align*} 
{\mbox Var}(\tilde S_{[m/2]})
&\le \sum_{l=N+1}^{[m/2]-1} \frac{1}{(2  l+2(1-\mu))^2 } +\frac{1}{(2l+2\mu)^2}.
\end{align*} 
As  $\frac{2x}{\pi}\le \sin x\le x$ for $0 \leq x \leq  \frac{\pi}{2}$, 
$$  \frac{\sin\frac{\pi}{m}}{\gamma_1^m\sin\frac{l}{m}\pi+\gamma_0^m\sin\frac{l+1}{m}\pi}\le\frac{\pi/2}{\gamma_1^ml+\gamma_0^m(l+1)},$$
and 
$$ \frac{\sin\frac{\pi}{m}}{\gamma_1^m\sin\frac{l+1}{m}\pi+\gamma_0^m\sin\frac{l}{m}\pi}\le\frac{\pi/2}{\gamma_1^m(l+1)+\gamma_0^m  l},$$
for $N \leq l \leq m/2-1$. Thus for any $\varepsilon>0$, there is a $N_0>0$ such that for all $m\ge 2(N_0+1)$, the right hand-side term of (\ref{third})  is smaller than $\varepsilon$.  
Hence $\sup_{1 \leq p \leq m/2}\vert S_p^m-S_p-\E(S_p^m-S_p)\vert$ converges in probability to $0$. Let us deal with the second term of the right hand-side of (\ref{first}).
 One has for $0\leq q \leq p\leq m/2$,
$$\vert \E(S_p^m-S_q^m)\vert \le \sum_{l=q}^{p-1}\frac{\pi^2(\gamma_0^m+\gamma_1^m)}{4(\gamma_1^ml+\gamma_0^m(l+1))(\gamma_1^m(l+1)+\gamma_0^ml)}$$
and $$\E(S_p-S_q)=\sum_{l=q}^{p-1}\frac{2(2\mu-1)}{(2  l+2(1-\mu))(2l+2\mu)}.$$
As for a fixed $q$,  $\E(S_q^m)$ converges to $\E(S_q)$ when $m$ goes to infinity,  one obtains that $\sup_{0\le p\le m/2}\vert \E(S_p^m-S_p)\vert $ converges to $0$, which finishes the proof of the lemma. \end{proof}

Recall that $\xi^m(\infty)$, resp.\ $\xi(\infty)$, are the Verma string parameters of $W^{(\frac{m}{\pi},\mu)}$, resp.\ $ B^{(\mu)}$.

\begin{prop} \label{convGamma} Suppose that  $\mu\ne 0,1$. Let   $(\lambda_m)$ be a sequence of $\R^2$ such that $\lambda_m\in C_m$  and such that $\tau_m\lambda_m$ tends to $\lambda$ when $m$ tends to $\infty$ where $\lambda\in C_{\mbox{aff}}$.   The random sets $\{\xi^m(\infty)\in \Gamma_m(\lambda_m)\}$  converge  in probability to $\{\xi(\infty)\in \Gamma(\lambda)\}$.
\end{prop}
\begin{proof} To prove convergence in probability of a sequence, it is enough to show that each subsequence has a subsequence which converges almost surely. Therefore, by Propositions \ref{convunifM} and  \ref{limite_L}, working with a subsequence, we can suppose that the set of $\omega\in \Omega$ for which, as $m\to +\infty$,
$$\sup_{1 \leq k \leq m}\|\tau_mM^m_k(\infty)(\omega)-M_k(\infty)(\omega)\| \to 0,$$ and for which $M_m(\infty)(\omega)$ tends to $ L^{(\mu)}(\infty)(\omega)$ has probability one. If
\begin{align*}X_k^m=M_m^m(\infty)-M_k^m(\infty),\quad X_k=L^{(\mu)}(\infty)-M_k(\infty),\end{align*}
then, when $m$ tends to $+\infty$ 
(see (\ref{tauv})) 
\begin{align}\label{lim_v}\sup_{0\le k\le m} |\tilde \alpha_k^m(X_{k}^{m}(\omega))-\tilde \alpha_k(X_{k}(\omega))|\to 0, \end{align} 
and
\begin{align}\label{lim2_v}\tilde \alpha_m(X_{m}(\omega))\to 0.\end{align} 
For these $\omega$ we will show that
\begin{align*}\limsup_{m\to \infty}\{\xi^m(\infty)\in \Gamma_m(\lambda_m)\}\subset  \{\xi(\infty)\in \Gamma(\lambda)\}\subset \liminf_{m\to\infty} \{\xi^m(\infty)\in \Gamma_m(\lambda_m)\}.\end{align*}
Notice that $\tilde \alpha_k^m(\lambda_m)$ tends to $\tilde \alpha_k(\lambda)$ for $k=0,1$ by Lemma \ref{lemm_tau}.   One has
$\xi^m(\infty)\in \Gamma_m(\lambda_m)$ if and only if  $ \tilde \alpha_k^m(\lambda_m-X^m_k)\ge 0$ for $ 0 \leq k<m,$ and
$\xi(\infty)\in \Gamma(\lambda)$ if and only if $ \tilde \alpha_k(\lambda -X_k)\ge 0$ for every  $k\in \N.$
The left inclusion above follows from (\ref{lim_v}).
Now, suppose that $\xi(\infty)(\omega)\in \Gamma(\lambda)$. Since $\lambda$ is fixed and since the distribution of $X_k$ is continuous, 
$$\mathbb P(\tilde\alpha_k(\lambda-X_k)=0)=0,$$ and one can suppose that for all $k\ge 0$ (see \ref{majalpha}), $$\tilde\alpha_k(\lambda-X_k(\omega))>0.$$
We choose $\varepsilon>0$ such that $\tilde \alpha_i(\lambda)>\varepsilon,  i=0,1.$ Using (\ref{lim_v}), one can choose $m_0$ such that $m\ge m_0$ implies  that
$$\sup_{0\le k\le m} \tilde \alpha_k^m(X_{k}^{m}(\omega))-\tilde \alpha_k(X_{k}(\omega))<\varepsilon.$$ 
Then, by (\ref{lim2_v}), we choose  $k_0$ such that $k\ge k_0$ implies that
$$\tilde \alpha_k(X_{k}(\omega))<\tilde \alpha_k(\lambda)-\varepsilon.$$ 
As for each $k$, $\tilde \alpha_k^m(X_k^m)(\omega)$ converges to $\tilde\alpha_k(X_k)(\omega)$ when $m$ goes to infinity, one takes
 $m_1$ such that  when $m\ge m_1$, $$\tilde \alpha_k^m(X_{k}^{m}(\omega))<\tilde \alpha_k(\lambda),$$ for $k\leq k_0$. For $m\ge \max(m_0, m_1)$ one has for $k\ge k_0$
$$\tilde \alpha_k^m(X_{k}^{m}(\omega))=\tilde \alpha_k^m(X_{k}^{m}(\omega))-\tilde \alpha_k(X_{k}(\omega))+\tilde \alpha_k(X_{k}(\omega))<\tilde \alpha_k(\lambda).$$  This proves the right inclusion above, and shows the proposition.
\end{proof}

This  proposition will be crucial to move from  Verma affine string parameters to affine string parameters in the next section.  

\section{The highest weight process $\Lambda^{(\mu)}$}\label{HWP}

In this section we introduce the highest weight process $\Lambda^{(\mu)}$ associated with the space-time Brownian motion $B^{(\mu)}$ as the limit of its Pitman's transforms (with a correction) and show that it coincides in distribution with $A^{(\mu)}$.
We define  for $k\ge 0$ and $t \in \R^+$,
$$M_k(t)=\frac{1}{2}\xi_k(t)\alpha_k+\sum_{n=0}^{k-1}\xi_n(t)\alpha_n,$$
where $\xi(t)=\{\xi_n(t), n\geq 0\}$ are the affine   string   parameters of  $ B^{(\mu)}$ on $[0,t]$, and
$$M_k^m(t)=\frac{1}{2}\xi_{k}^m(t)\alpha_{k}^m+\sum_{n=0}^{k-1}\xi_n^m(t)\alpha_n^m,$$
where $\xi^m(t)=\{\xi_n^m(t), 0 \leq n < m\}$ are the   string   parameters of  $W^{(\frac{m}{\pi},\mu)}$ on $[0,t]$ (by convention $\xi_n^m(t)=0$ when $n \geq m$). In particular, \begin{align}\label{Mmt}M_m^m(t)=\sum_{n=0}^{m-1}\xi_n^m(t)\alpha_n^m= \PP^m_{w_m}W^{(\frac{m}{\pi},\mu)}(t)-W^{(\frac{m}{\pi},\mu)}(t)
.\end{align}
Below we will often use the Cameron Martin Girsanov  theorem (see \cite{kara} Theorem 3.5.1), abreviated CMG,  to reduce the case $0 \leq\mu \leq 1$ to the case where $0 <\mu < 1$. Each time the reason is the following: on the time interval $[0,t]$ the processes considered for $\mu=0$ and $\mu=1$ are absolutely continuous with the ones for $0 <\mu < 1$.
\begin{prop} \label{convMproba} For each $t \geq 0 $,  $M_k(t)$ converges in probability  when $k$ goes to infinity. We denote by $L^{(\mu)}(t)$ the limit. 
\end{prop}
\begin{proof}  By CMG, one can suppose that $\mu\ne 0,1$.  For every $k\ge 1$ and $t \geq 0$, 
  $\tau_mM_k^m(t)$ converges to $M_k(t)$ almost surely when $m$ goes to infinity by Proposition \ref{conv_ps} and Lemma \ref{lemm_tau}.  Let $\varepsilon>0$. 
For $p,q\in \N$, 
\begin{align*}
\P(\| M_{p+q}(t)-M_p(t)\| \geq \varepsilon)&=\lim_{m \to \infty}\P(\| \tau_m M_{p+q}^m(t)-\tau_m M^m_p(t)\|\geq\varepsilon)\\ &=\lim_{m \to \infty}\E(1_{\| \tau_m M_{p+q}^m(t)-\tau_m M^m_p(t)\|\ge \epsilon}\vert  \sigma(\Lambda^{(\mu)}_m(s), s \leq t))\\
&=\lim_{m\to \infty}f_m( \Lambda^{(\mu)}_m(t)),
\end{align*}
where $\Lambda^{(\mu)}_m(t)=\PP^m_{w_m}W^{(\frac{m}{\pi},\mu)}(t)$  and where  $$f_m(\lambda)=\P(\| \tau_m M_{p+q}^m(\infty)-\tau_m M^m_p(\infty)\| \geq \varepsilon | \xi^m(\infty)\in\Gamma_m(\lambda)),$$
by Corollary \ref{condxi}. If $\lambda\in C_{\mbox{aff}}$ and $\tau_m\lambda_m$ tends to $\lambda$, then by  Propositions \ref{convGamma} and \ref{limite_m}, $f_m(\lambda_m)$ tends to  $f(\lambda)$ where
$$f(\lambda)=\P(\| M_{p+q}(\infty)-M_p(\infty)\| \geq \varepsilon | \xi(\infty)\in\Gamma(\lambda)).$$ Notice that this is well defined by Proposition \ref{positif}. On the other hand we have seen in Theorem \ref{thmconvTau} that when $m$ tends to infinity,
$\tau_m \Lambda^{(\mu)}_m(t)$ converges in distribution to $A^{(\mu)}(t)$, where $A^{(\mu)}$ is the Brownian motion conditioned to remain in $C_{\mbox{aff}}$ defined in Section \ref{AppAffi}.
 Therefore $\E(f_m(\Lambda^{(\mu)}_m(t)))$ tends to $ \E(f(A^{(\mu)}(t)))$. This shows that
\begin{align*}\P(\| M_{p+q}(t)&-M_p(t)\| \geq \varepsilon) \\ =&\int_\Omega \P(\| M_{p+q}(\infty)-M_p(\infty)\|\geq \varepsilon\vert \xi(\infty)\in\Gamma(A^{(\mu)}(t)(\omega)))d\P(\omega).\end{align*}
Since $M_p(\infty)$ converges a.s.\ to $L^{(\mu)}(\infty)$ (Proposition \ref{limite_L}) we see that $M_p(t), p\in \N,$ is a Cauchy sequence for the convergence in probability, and thus converges in probability. 
 \end{proof}

 We will see in Theorem \ref{convM} that the convergence in the proposition actually also holds almost surely. 
 
 \begin{prop}\label{limL}  When $\mu\ne 0,1,$ there is an almost surely finite random time $\mathfrak t \geq 0$ such that $L^{(\mu)}(t)= L^{(\mu)}(\infty)$ for $t \geq \mathfrak t.$\end{prop}

\begin{proof} As in the proof of Proposition \ref{x_fini} we define $\mathfrak t_{0}=\max\{t \geq 0,  B_t^{(\mu)} \not \in  C_{\mbox{aff}}\}, $ and, for $n \geq 0$,
$$\mathfrak t_{n+1}=\max\{t \geq \mathfrak t_n,  B_t^{(\mu)}+\sum_{k=0}^n\xi_k(t) \alpha_k \not \in  C_{\mbox{aff}}\}. $$
 We see by induction that  $\mathfrak t_n < +\infty$ and that for $t > \mathfrak t_n$, $\xi_n(t)=\xi_n(\infty)$. Therefore
 $$\mathfrak t_{n+1}=\max\{t \geq \mathfrak t_n,  B_t^{(\mu)}+\sum_{k=0}^n\xi_k(\infty) \alpha_k \not \in  C_{\mbox{aff}}\}. $$ 
We know that a.s., $\sum_{k=0}^n\xi_k(\infty) \alpha_k$ is bounded. Hence $\mathfrak t=\sup{\mathfrak t_n} < +\infty$  and for  $t \geq \mathfrak t,$ $\xi_k(t)=\xi_k(\infty)$ for all $k \geq 0$. So, for $t \geq\mathfrak t$, $L^{(\mu)}(t)=L^{(\mu)}(\infty)$. 
\end{proof}

\begin{prop}\label{mmm}  For any $t \geq0$, as $m $ tends to $ \infty$, $\tau_mM_m^m(t)$ converges  in probability to $  L^{(\mu)}(t)$. 
\end{prop}

\begin{proof}  One can suppose that $\mu\ne 0,1$. In that case, one has for $\varepsilon>0$,
\begin{align*}
\P(\|  \tau_mM_m^m(t)-L^{(\mu)}(t)\|>\varepsilon)   & \le   \P(\| \tau_mM_m^m(t)-M_m(t))\|>\frac{\varepsilon}{2})\\
&\quad \quad +\P(\| M_m(t)-L^{(\mu)}(t)\|>\frac{\varepsilon}{2}).
\end{align*}
For the first term, we condition by $\sigma(\Lambda^{(\mu)}_m(s), s \leq t)$  as in the proof of Proposition \ref{convMproba} and we use Proposition \ref{convunifM}.  Convergence in probability of Proposition \ref{convMproba} gives the convergence of the second term. 
\end{proof}
 
We introduce the following definition, \begin{defi} We define the  highest weight process $\Lambda^{(\mu)}$ of $B^{(\mu)}$ by, for $t \geq 0$, 
   $$\Lambda^{(\mu)}(t)=B_t^{(\mu)}+L^{(\mu)}(t).$$
\end{defi}
In the analogy with the Littelmann model, for $t>0$ fixed, $\{B^{(\mu)}_s, 0 \leq s \leq t\}$ can be considered as a path with weight $B^{(\mu)}_t$, 
with highest weight $\Lambda^{(\mu)}(t)$, and $L^{(\mu)}(t)$ as the weight seen from the highest weight. Recall that
\begin{align}\label{egalitLM}\Lambda^{(\mu)}_m(t)= \PP^m_{w_m}W^{(\frac{m}{\pi},\mu)}(t)=W^{(\frac{m}{\pi},\mu)}_t+M_m^{m}(t).\end{align}

\begin{prop}\label{convcouple}In probability, 
\begin{align*}&\lim_{m \to +\infty}\tau_m W^{(\frac{m}{\pi},\mu)}_t=B_t^{(\mu)},\\ &\lim_{m \to +\infty}\tau_m\Lambda^{(\mu)}_m(t)=\Lambda^{(\mu)}(t).\end{align*}\end{prop}
\begin{proof}The first statement is obvious. The second one is then a consequence of Proposition \ref{mmm} and (\ref{egalitLM}). 
\end{proof}

Let $\{A^{(\mu)}(t),t\ge 0\}$ be the conditioned space-time Brownian motion in the affine Weyl chamber $C_{\mbox{aff}}$ with drift $\mu$ starting from the origin, defined in \ref{AppAffi}.

\begin{theo} \label{lim_mu}  In distribution, $$\{\Lambda^{(\mu)}(t),t \geq 0\}=\{A^{(\mu)}(t),t \geq 0\}.$$  
\end{theo}
\begin{proof} This follows from Proposition \ref{convcouple}, 
and Theorem \ref{thmconvTau}.  
\end{proof}

This shows in particular that $\Lambda^{(\mu)}$ has a continuous version. 
Recall that for any $\lambda \in C_{\mbox{aff}}$, $\P(\xi(\infty)\in \Gamma(\lambda)) >0$ by Proposition \ref{positif}.
\begin{prop}  \label{cond_couple} 
For $\mu\ne 0,1,$ and  $f:\R^2\times \Gamma \times \R^2 \to \R$, bounded and measurable,
$$\E(f(B_t^{(\mu)},\xi(t), \Lambda^{(\mu)}(t)) | \sigma(\Lambda^{(\mu)}(s),s \leq t))=g(\Lambda^{(\mu)}(t))$$
for each $t \geq 0$, where
$$g(\lambda)=\E(f(\lambda-L^{(\mu)}(\infty),\xi(\infty),\lambda) | \xi(\infty)\in \Gamma(\lambda)).$$
\end{prop}
\begin{proof}  It is enough to prove this when $f$ is a bounded continuous function depending on a finite number of variables.  By Corollary \ref{condxi},
\begin{align*}\E(f(\tau_m\Lambda^{(\mu)}_m(t)-\tau_mM_m^m(t),\xi^m(t), \Lambda^{(\mu)}_m(t)) |
\sigma(\Lambda^{(\mu)}_m(s),s\le t)))=g_m(\Lambda^{(\mu)}_m(t))
\end{align*}
where 
\begin{align*}g_m(\lambda)=\E(f(\tau_m\lambda-\tau_mM_m^m(\infty),\xi^m(\infty),\lambda) | \xi^m(\infty)\in\Gamma_m(\lambda)).\end{align*}
Let  $h$ be a bounded continuous function depending only on a finite number of variables.
Using the convergences in probability of the propositions \ref{mmm} and \ref{convcouple}, \begin{align*}&\E(f(\Lambda^{(\mu)}(t)-L^{(\mu)}(t),\xi(t), \Lambda^{(\mu)}(t))h(\Lambda^{(\mu)}(s),s\le t)))\\
&=\lim_{m \to +\infty}\E(f(\tau_m\Lambda^{(\mu)}_m(t)-\tau_mM_m^m(t),\xi^m(t), \tau_m\Lambda^{(\mu)}_m(t))
h(\tau_m\Lambda^{(\mu)}_m(s),s\le t)))\\
&=\lim_{m \to +\infty}\E(g_m(\Lambda^{(\mu)}_m(t))
h(\tau_m\Lambda^{(\mu)}_m(s),s\le t))).\end{align*}
 Since $\tau_m\Lambda^{(\mu)}_m(t)$ tends to $\Lambda^{(\mu)}(t)$, we also know from Proposition \ref{convGamma} that  $\{\xi^m(\infty)\in \Gamma_m(\Lambda^{(\mu)}_m(t))\}$ tends to $\{\xi(\infty)\in \Gamma(\Lambda^{(\mu)}(t))\}$ in probability. Therefore the limit above equals 
$$\E(g(\Lambda^{(\mu)}(t))
h(\Lambda^{(\mu)}(s),s\le t)))$$
which proves the proposition.
\end{proof}

 Applying the proposition with
 $f(b,\xi, \lambda)=1_{\Gamma(\lambda)}(\xi)$ we obtain
\begin{coro} Almost surely, $\xi(t)$ belongs to  $\Gamma(\Lambda^{(\mu)}(t))$.
\end{coro}

\begin{theo}\label{convM}
 $M_k(t)$ converges almost surely to $L^{(\mu)}(t)$ when $k \to +\infty$.
 \end{theo}
\begin{proof}
 One can suppose that $\mu\ne 0,1$. In that case,   by Proposition \ref{cond_couple}, one has 
$$\E(1_{\{\lim M_k(t)=L^{(\mu)}(t)\}}| \sigma(\Lambda^{(\mu)}(s), s \leq t))=\E(1_{\{\lim M_k(\infty)=L^{(\mu)}(\infty)\}}| \xi(\infty)\in \Gamma(\Lambda^{(\mu)}(t))).$$ As the right hand-side term equals $1$ by Proposition \ref{limite_L}, one obtains that
$$\P(\lim M_k(t)=L^{(\mu)}(t))=\E(\E(1_{\{\lim M_k(t)=L^{(\mu)}(t)\}}| \sigma(\Lambda^{(\mu)}(s), s \leq t)))=1.$$
\end{proof}

\section{The representation theorem using Pitman and L\'evy transforms}  \label{LP}

\subsection{Representation of the conditioned space-time Brownian motion $A^{(\mu)}$} Let us remind where we stand. For $B_t^{(\mu)}=(t,B_t+t\mu)$
we have written $$\PP_{s_{n}}\cdots \PP_{s_{1}}\PP_{s_{0}} B^{(\mu)}(t)=B_t^{(\mu)}+\sum_{i=0}^{n}\xi_i(t)\alpha_i,$$
see (\ref{formsum}), and 
$$M_{n+1}(t)=\frac{1}{2}\xi_{n+1}(t)\alpha_{n+1}+\sum_{i=0}^{n}\xi_i(t)\alpha_i.$$
We have seen that when for $t >0$,
$\lim_{k\to  +\infty} M_k(t)=L^{(\mu)}(t)$ a.s.\ and that the process $\Lambda^{(\mu)}(t)=B^{(\mu)}_t+L^{(\mu)}(t)$, $t\ge 0,$ has the same distribution as the process
 $A^{(\mu)}(t), t \geq 0$. Hence,
\begin{theo}\label{thmsans} For $ t\ge 0$, almost surely, 
$$\lim_{n\to +\infty}\PP_{s_{n}}\cdots \PP_{s_{1}}\PP_{s_{0}} B^{(\mu)}(t)+\frac{1}{2}\xi_{n+1}(t)\alpha_{n+1}$$
exists, and the limiting process has the same distribution as $\{A^{(\mu)}(t),t \geq 0\}$.
\end{theo}
To interpret the correction term, it is worthwhile to introduce the L\'evy transform (sometimes called the Skorokhod transform). A theorem of L\'evy (see Revuz and Yor \cite{revuzyor}, VI.2) states that if $\beta$ is the standard Brownian motion, then
$$\mathcal L\beta(t)=\beta_t-\inf_{0 \leq s \leq t}\beta_s, t \geq 0,$$
has the same distribution as $|\beta_t|, t \geq 0,$ and that $-\inf_{0 \leq s \leq t}\beta_s$ is the local time of $\mathcal L\beta$ at $0$. We define here the following L\'evy transform (sometimes the L\'evy transform of $\beta$ is defined as $\int_0^t \mbox{sign}(\beta_s) d\beta_s$, this is related to  the transform here, but different).
\begin{defi}  For $\eta \in {\mathcal C}_0(\R^2)$ and $i=0,1$, the L\'evy transform $\mathcal L_{s_i}\eta$ of $\eta$ is
\begin{align*}
\mathcal L_{s_i}\eta(t)&=\eta(t)-\frac{1}{2}\inf_{0\leq s\le t}\tilde \alpha_i(\eta(s))\alpha_i.
\end{align*}
\end{defi}
Another way to state the former theorem is
\begin{theobis}{thmsans} For $t \geq 0$, 
$$\lim_{n \to +\infty} \mathcal L_{s_{n+1}}\PP_{s_{n}}\cdots \PP_{s_{1}}\PP_{s_{0}} B^{(\mu)}(t)$$
exists a.s.\ and the limiting process has the same distribution as $A^{(\mu)}$.
\end{theobis}

The following proposition indicates that the presence of the L\'evy transform is due to the bad behavior of $A^{(\mu)}(t)$ for  $t$ near 0.

 \begin{prop} \label{conv-xi} For all $t > 0$, a.s.
 $\lim_{k \to +\infty}\xi_k(t)=2.$
 \end{prop}
 \begin{proof} When $\mu\ne 0,1$, $\xi_k(\infty)$ tends to $2$ almost surely (Theorem \ref{lim2}), and the proposition follows from Proposition \ref{cond_couple} by conditioning by $\sigma(\Lambda^{(\mu)}(s), s \leq t)$. The case when $\mu\in\{0,1\}$ follows from the CMG theorem.
 \end{proof}
This implies that, for $t >0$,
$$\lim_{n\to +\infty}\PP_{s_{n}}\cdots \PP_{s_{1}}\PP_{s_{0}} B^{(\mu)}(t)+(-1)^{n} \alpha_1=A^{(\mu)}(t).$$
So without the Levy correction, the iterates of Pitman's transform do not converge.

\medskip

In the whole paper we have chosen to begin by first applying  $\PP_{s_0}$ to $B^{(\mu)}$ and then $\PP_{s_1}$, in the iterations of Pitman's transforms. Let us show the non trivial fact that we obtain the same limit if we begin with $\PP_{s_1}$. More precisely, if we denote with a tilde the  quantities previously defined when we begin by $\PP_{s_1}$ instead of $\PP_{s_0}$, one has,
\begin{theo}\label{les2}
(i) Almost surely, 
$$\lim_{n\to +\infty}\mathcal L_{s_{n+1}}\PP_{s_{n}}\cdots \PP_{s_{1}}\PP_{s_{0}} B^{(\mu)}(t)=\lim_{n\to +\infty}\mathcal L_{s_{n+1}}\PP_{s_{n}}\cdots \PP_{s_{2}}\PP_{s_{1}} B^{(\mu)}(t).$$

(ii) $\tilde \xi(t)$ defined for $B^{(\mu)}$ has the same distribution as $\xi(t)$ defined for $B^{(1-\mu)}$.
\end{theo}

\begin{proof}
For (i) we have to prove that $L^{(\mu)}(t)=\tilde L^{(\mu)}(t)$. We have seen in Proposition \ref{mmm} that 
$\tau_m M_m^m(t)$ converges in probability to $L^{(\mu)}(t)$. By the same proof $\tau_m \tilde M_m^m(t)$ converges in probability to $\tilde L^{(\mu)}(t)$. 
Now by (\ref{egalitLM}), 
$$ M_m^{m}(t)=\PP^m_{w_m}W^{(\frac{m}{\pi},\mu)}(t)-W^{(\frac{m}{\pi},\mu)}_t$$
and by Theorem \ref{braidP} 
$$\PP^m_{w_m}W^{(\frac{m}{\pi},\mu)}=\tilde \PP^m_{w_m}W^{(\frac{m}{\pi},\mu)}.$$
since 
$$w_m= s_{m-1}^m\cdots s_1^ms_0^m= s_m^m\cdots s_2^ms_1^m.$$
Hence $M_m^m(t)=\tilde M_m^m(t)$ and $L^{(\mu)}(t)=\tilde L^{(\mu)}(t)$ a.s.. 
(ii) follows from the relation (\ref{entrela}) and the fact that $TB^{(\mu)}$ has the same distribution as  $B^{(1-\mu)}$.
\end{proof}

\subsection{Representation of the conditioned Brownian motion in $[0,1]$ }\label{results}

For a continuous real path $\varphi:\R_+\to \R$ such that $\varphi(0)=0$, we have defined 
\begin{align*}\mathcal L_0\varphi(t)&=\varphi(t)+\inf_{0\leq s\le t}(s-\varphi(s)),\, 
\PP_0\varphi(t)=\varphi(t)+2\inf_{0\leq s\le t}(s-\varphi(s)), \\\mathcal L_1\varphi(t)&=\varphi(t)-\inf_{0\leq s\le t}\varphi(s),\, 
\PP_1\varphi(t)=\varphi(t)-2\inf_{0\leq s\le t}\varphi(s).\end{align*}
For $B^\mu_t=B_t+t\mu, t \geq 0$, the theorem stated in the introduction is 
\begin{theo} \label{thmprincipal}
For any $t> 0$, almost surely,
$$\lim_{n\to \infty}   t\mathcal L_{n+1} \PP_{n}\cdots \PP_1 \PP_0 B^{\mu}(1/t)=\lim_{n\to \infty}   t\mathcal L_{n+1} \PP_{n}\cdots \PP_2 \PP_1 B^{\mu}(1/t)=Z^\mu_t,$$ 
where $Z^\mu_t$, $t\ge 0,$ is a Brownian motion conditioned to stay in the interval $[0,1]$ forever, starting from $\mu$.
\end{theo}

\begin{proof} The proof is just the juxtaposition of the theorems \ref{thmsans}, \ref{les2} and \ref{bigthm2}.
\end{proof}
Remark that, for any $t >0$ and a.s., by this theorem and Proposition \ref{conv-xi},
$$\lim_{n\to \infty}   t(\mathcal  \PP_{n}\cdots \PP_1 \PP_0 B^{\mu}(1/t)+(-1)^n2)=Z^\mu_t.$$
 As an illustration let us show that:
 if we write
$L^{(\mu)}(\infty)=(0,D^{\mu}(\infty))$
for some $D^{\mu}(\infty)\in \R$,
 \begin{prop} \label{Zpetit} When $\mu\ne 0,1,$ there is a standard Brownian motion $\beta$ and a stopping time $\tau$ with respect to the filtration $\{\sigma(\beta_s, s \leq t), t \geq 0 \}$  such that $\tau >0$ a.s.\ and for $0\leq t \leq \tau$,
$$Z^\mu_t=\beta_t+\mu+ tD^{\mu}(\infty).$$
\end{prop}
\begin{proof} Let $\mathfrak t\geq 0$ be the random time given by Proposition \ref{limL}. For $t >\mathfrak t$, $L^{(\mu)}(t)=L^{(\mu)}(\infty)$ hence  $A^{(\mu)}(t)= B_t ^{(\mu)}+L^{(\mu)}(\infty)$. Thisx implies the relation for $\tau=1/\mathfrak t$ and $\beta_t=tB_{1/t}$, since $(t,tZ^\mu_{1/t})=A^{(\mu)}(t)$. 
\end{proof}
Notice that, by (\ref{espD}), $\E(D^{\mu}(\infty))=\pi \cot \pi \mu$  as expected from the expression of the drift of the generator $\frac{1}{2}d^2/dx^2+(\pi \cot \pi x)d/dx$ of $Z^\mu_t$.

\subsection{A property of Pitman transform for piecewise $C^1$ paths}
As mentioned in the introduction, the need of an infinite number of Pitman transforms to represent the space-time Brownian motion in $C_{\mbox{aff}}$ is due to its wild behaviour and in particular to its non differentiability. The situation is much simpler for regular space-time curves, as shown by the following proposition (recall that paths in $\bar C_{\mbox{aff}}$ are fixed under the Pitman transforms).

\begin{prop}\label{nombrefini} Let $ \varphi:[0,T]\to \R$  be a continuous piecewise $C^1$ function such that $ \varphi(0)=0$ and let $\eta(t)=(t, \varphi(t))$.
There is 
an $n$ such that for all $t\in [0,T]$, 
\begin{align}\label{pathC1}\PP_{s_n}\cdots \PP_{s_1}\PP_{s_0} \eta(t)\in \bar C_{\mbox{aff}},\end{align}
and for all $m \geq n$, 
\begin{align}\label{mplusgrand}\PP_{s_m}\cdots \PP_{s_1}\PP_{s_0} \eta(t)=\PP_{s_m}\cdots \PP_{s_1}\PP_{s_0} \eta(t).\end{align}
\end{prop}
We use the notations of 1.2 and Section \ref{results}. It is equivalent to proving that 
there is an $n>0$  such that
$$ 0 \leq \PP_{n}\cdots \PP_{1}\PP_{0}  \varphi(t) \leq t$$
for all $t\in [0,T]$. 
Let $\tau_1( \varphi)=\inf\{t > 0;  \varphi(t)<0\}$, $\tau_0( \varphi)=\inf\{t > 0;  \varphi(t)>t\}$,  and let 
$| \varphi'|$ be the supremum of the left and right derivatives of $ \varphi$ on $[0,T]$.
\begin{lemm}\label{lem1}
(1) $|(\PP_1 \varphi)'| \leq | \varphi'|$,  and  $|(\PP_0 \varphi)'| \leq 2+| \varphi'|.$

(2) When $0 \leq t \leq \tau_1( \varphi)\wedge \tau_0( \varphi)$,
$\PP_1  \varphi(t)=\PP_0  \varphi(t)= \varphi(t). $

(3) There is an $n >0$ such that
$0 \leq (\PP_{n}\cdots \PP_{1} \PP_{0}  \varphi)'(0) \leq 1.$ 

\end{lemm}

\begin{proof} (1) and (2) are straightforward. For (3): if $\sigma_0$ and $\sigma_1$ are the real reflections given by $\sigma_0(x)=-x$ and $\sigma_1(x)=2-x$, then
$(\PP_i \varphi)'(0)=\sigma_i (\varphi'(0)),$ for $i=0,1$ and it is well known that one can bring any real number into $[0,1]$ by repeated actions of $\sigma_0$ and $\sigma_1$.
\end{proof}

\begin{lemm}\label{lem2} Proposition \ref{nombrefini} holds when
$\tau_1( \varphi)\wedge \tau_0( \varphi)  >0$. \end{lemm}

\begin{proof} We first suppose that  $\tau_0( \varphi) < \tau_1( \varphi) $. Let $ \varphi_1=\PP_0 \varphi,  \varphi_2=\PP_1 \varphi_1,  \varphi_3=\PP_0 \varphi_2,\cdots$ and $a_1=\tau_0( \varphi), a_2=\tau_1( \varphi_1), a_3=\tau_0( \varphi_2), \cdots$.
By (1) of Lemma \ref{lem1}, 
$| \varphi'_n| \leq | \varphi'|+2n$.
Since $ \varphi_{2n}(a_{2n})=0,$ and $ \varphi_{2n}(a_{2n+1})=a_{2n+1},$ we obtain by the mean value theorem that $$a_{2n+1}-a_{2n} \geq \frac{a_{2n+1}}{| \varphi'|+4n}.$$
So the increasing sequence $a_n$ is bigger that $T$ for $n$ large enough. When $\tau_0( \varphi) > \tau_1( \varphi) $, then $\tau_0( \varphi_1) > \tau_1( \varphi_1) $ and we use the same proof for $ \varphi_1$ instead of $\varphi$ by beginning to apply $\PP_1$.
\end{proof}

\begin{proof}[Proof of Proposition \ref{nombrefini}] Using (3) of  Lemma \ref{lem1} one can suppose that $0\leq   \varphi'(0) \leq 1$. When $0<  \varphi'(0) < 1$, then $\tau_1( \varphi)\wedge \tau_0( \varphi)  >0$ and we can apply Lemma \ref{lem2}. If $ \varphi'(0)=1$, let $\gamma=\PP_0 \varphi$. Then $\gamma'(0)=1$ and $\gamma(t) \leq t$ for $t\geq 0$. In that case $\tau_0(\gamma)\geq T$ and $\tau_1(\gamma)>0$. Indeed otherwise there is a sequence $t_n>0$ decreasing to $0$ such that $\gamma(t_n) \leq 0$ and $\gamma'(0)$ cannot be $1$. So we can also apply Lemma \ref{lem2}. The case $ \varphi'(0)=0$ is similar. This proves (\ref{pathC1}), which gives  \ref{mplusgrand} since the action of the path transforms $\PP_{s_0}$ and $\PP_{s_1}$ are trivial for path in $\bar C_{\mbox{aff}}$.\end{proof}

\section{Some probability distributions}  \label{sec_distributions}

\subsection{The   conditional distribution of $B^{(\mu)}_t$ given $\sigma(\Lambda^{(\mu)}(s),s\le t)$  }\label{condition}

We will compute this conditional distribution by approaching it by the dihedral case.

Let $m \geq 1$. Recall the definitions of $\psi_{\gamma}^m(v)$ and $h_m$  given in (\ref{psigamma}) and (\ref{hm}).
 Let $W^{(\gamma)}$ be the standard planar Brownian motion in $\R^2$ with drift $\gamma \in \bar C_m$.
\begin{lemm}\label{lem_cond}
For $\zeta\in \R^2$ and $v=\PP^m_{w_m}W^{(\gamma)}(t),$
$$\E(e^{\langle \zeta,W^{(\gamma)}_t \rangle}\vert \sigma(\PP^m_{w_m}W^{(\gamma)}(s), 0 \leq s \leq t))=\frac{\psi_{\zeta+\gamma}^m(v)}{\psi_{\gamma}^m(v)}\frac{h_m(\gamma)}{h_m(\zeta+\gamma)}e^{\langle \zeta,v\rangle}.$$
\end{lemm} 
\begin{proof} Theorem 5.5 in \cite{bbo2} gives \begin{align*}
\E(e^{\langle \zeta,W_t^{(0)}\rangle}\vert  \sigma(\PP^m_{w_m}W^{(0)}(s), 0 \leq s \leq t))= C\frac{\psi_\zeta^m(v)}{h_m(v)h_m(\zeta)}e^{\langle \zeta,v\rangle},
\end{align*}
  for $v=\PP^m_{w_m}W^{(0)}(t)$ and  a   constant $C>0$ independent of $v$ and $\zeta$.
We conclude with Proposition \ref{prop_bayes}.\end{proof}

For any $l\in \R$ and $(t,x)\in \bar C_{\mbox{aff}}$ we define $\varphi_l$ as in  (\ref{defphi}) by
\begin{align}\varphi_l(t,x)=
\frac{e^{-l x}}{\sin(l\pi )}
\sum_{k\in \Z}{ \sinh(2ktl+xl)}e^{-2(kx+k^2t)}.
\end{align} 

\begin{lemm}\label{limpourDH} Let $\mu_m \in \R^2,  v_m \in \bar C_m$ be such that 
$$\lim_{m \to +\infty}\tau_mv_m=(t,x), \lim_{m \to +\infty}\tau_m\mu_m=(1,\mu),$$
 then 
\begin{align}\lim_{m\to \infty}\psi_{\mu_m}^m(v_m)=\frac{\sin(\mu\pi)}{2}\varphi_{\mu}(t,x),\label{limpsi}\\
\lim_{m\to \infty}(\frac{\pi}{m})^m{h_m(\mu_m)}=\sin(\mu\pi)\label{limh}.\end{align}
\end{lemm} 

\begin{proof}  Let $r$ be the rotation of $\R^2$ of angle $2\pi/m$ and $s$ be the symmetry $s(a,b)=(a,-b)$, for $(a,b)\in \R^2$. 
Let $J(m)=\{-[m/2],\cdots, [m/2]\}$ when $m$ is odd and  $J(m)=\{-[m/2]+1,\cdots, [m/2]-1\} \cup \{m/2\}$ when $m$ is even.
The dihedral group $I(m)$ is $I(m)=\{r^k, k\in J(m)\}\rtimes\{Id,s\},$
$l(r)=2$, and $l(s)=1$. Let $\mu_m=(\mu_m(1),\mu_m(2))$ and $v_m=(v_m(1),v_m(2))$. One has
\begin{align*}\psi_{\mu_m}^m(v_m)&=e^{-\langle\gamma_m,\alpha_m\rangle}\sum_{k\in J(m)}(e^{\langle r^k(\mu_m),v_m\rangle}-e^{\langle r^ks(\mu_m),v_m\rangle})\\& =2e^{-\mu_m(2) v_m(2)}\sum_{k\in J(m)} I(m,k)
\end{align*}
where
\begin{align*}I(m,k)=&\exp[{\mu_m(1)v_m(1)(\cos{\frac{2k \pi}{m}}-1)+\mu_m(1)v_m(2)\sin{\frac{2k \pi}{m}}}]\\
&\times \sinh[\mu_m(2)v_m(2)\cos{\frac{2k \pi}{m}}-\mu_m(2)v_m(1)\sin{\frac{2k \pi}{m})}].\end{align*}
For $\varepsilon>0$, we choose $m_0 \geq 0$  such that, for $m \geq m_0$,
$$\sum_{m/4 \leq |k| \leq {m/2}}|I(m,k)| \leq \varepsilon,$$
by using the inequality $\cos(\frac{2k}{m}\pi)\leq 0$ when   $\frac{m}{4} \leq |k| \leq \frac{m}{2} $. 
Besides,  when $t\in[-\frac{\pi}{2},\frac{\pi}{2}]$, then
$ \cos t\le 1-\frac{t^2}{\pi}$ and $|\sin(t)| \leq |t|$, 
hence, for  $|k| \leq \frac{m}{4}$, 
$$|I(m,k)|\leq e^{-c_1 k^2+c_2 k}$$
where $c_1, c_2 >0$ do not depend on $m$. So, one can choose $N$ such that for $m\ge 1$
$$\sum_{  N \leq |k| \leq m/4} |I(m,k)|\leq \epsilon.$$
We obtain $(\ref{limpsi})$ since, for $N$ fixed, 
$$\lim_{m \to +\infty}\sum_{k=-N}^N I(m,k)=e^{-\mu x}\sum_{k=-N}^{N} \sinh(\mu(2kt+x))e^{-2(kx+k^2t)}.$$
The relation (\ref{limh}) is immediate.
\end{proof}
 
 For  $\lambda =(t,x)\in \bar C_{\mbox{aff}}$, we consider the probability measure $\nu_\lambda^{(\mu)}$ on $\R^2$ carried by the line $\{(t,y), y \in \R\}$ which has the Laplace transform  \begin{align}\label{proba_dh}\int e^{\langle \zeta, v\rangle} d\nu_\lambda^{(\mu)}(v) =\frac{e^{\langle \zeta, \lambda \rangle}\varphi_{\zeta_2+\mu}(\lambda) }{\varphi_{\mu}(\lambda) },\end{align}
 for $\zeta=(\zeta_1,\zeta_2) \in \R^2$. This measure  was  introduced by Frenkel \cite{fre} in his study of orbital measures (see below Section \ref{DH}).
\begin{theo}\label{DHdrift}
The conditional distribution of $B_t^{(\mu)}$ given $\sigma(\Lambda^{(\mu)} (s), s \leq t))$ is the probability measure $\nu_{\Lambda^{(\mu)} (t)}^{(\mu)}$.
\end{theo}

\begin{proof}  By Lemma \ref{lem_cond} and for $\mu_m=(\frac{m}{\pi},\mu)$ one has

$$\E(e^{\langle \zeta,\tau_mW^{(\mu_m)}_t \rangle}\vert   \sigma(\Lambda_m^{(\mu)} (s), s \leq t))=\frac{\psi_{\tau_m\zeta+\mu_m}^m(\Lambda_m^{(\mu)} (t))}{\psi_{\mu_m}^m(\Lambda_m^{(\mu)} (t))}\frac{h_m(\mu_m)}{h_m(\tau_m\zeta+\mu_m)}e^{\langle \tau_m\zeta,\Lambda_m^{(\mu)} (t)\rangle}.$$
Hence the formula $$\E(e^{\langle \zeta, B_t^{(\mu)}\rangle }\vert  \sigma(\Lambda^{(\mu)} (s), s \leq t))=\frac{\varphi_{\zeta_2+\mu}(\Lambda^{(\mu)} (t)) }{\varphi_{\mu}(\Lambda^{(\mu)}(t)) }e^{\langle \zeta, \Lambda^{(\mu)} (t)\rangle}$$ is
obtained by letting $m$ go to infinity  using Proposition  \ref{convcouple} and Lemma \ref{limpourDH}. 
 \end{proof}

\subsection{Distribution  of $L^{(\mu)}(\infty)$}\label{dmu} We suppose that $\mu\ne 0,1,$ and write $L^{(\mu)}(\infty)=(0,D^{\mu}(\infty))$.
The Laplace transform of
  $D^{\mu}(\infty)$
is, by (\ref{Dmu}), for $\tau >0$,
$$\E(e^{-\tau  D^{\mu}(\infty)})=\prod_{n=0}^\infty( ({1+\frac{\tau }{(n+\mu)}})( {1-\frac{\tau }{(n+1-\mu)}}))^{-1}.$$
Using the relation
\begin{align}\label{gradform}\frac{\Gamma(\alpha)\Gamma(\beta)}{\Gamma(\alpha+\gamma)\Gamma(\beta-\gamma))}=\prod_{n=0}^{+\infty}(1+\frac{\gamma}{n+\alpha})(1-\frac{\gamma}{n+\beta}),\end{align} (Formula  8.325.1 of  \cite{grad}) and
$ \Gamma (z)\Gamma (1-z)={ {\pi }/{\sin {(\pi z)}}}$,
we obtain that 
\begin{align} \label{LaplaceDmu}
\E(e^{-\tau  D^{\mu}(\infty)})=\frac{\sin(\pi \mu)}{{\sin(\pi (\mu+\tau ))}}.
\end{align}
In particular, \begin{align}\label{espD}\E(D^{\mu}(\infty))=\pi \cot (\pi\mu).\end{align}

\begin{coro} \label{Dundemi}
The density of $D^{1/2}(\infty)$ is
$1/(2\pi \cosh (x/2))$.
\end{coro}
\begin{proof}
One uses that the Fourier transform of  $1/(\pi\cosh x)$ is $1/\cosh(\lambda \pi/2)$.
\end{proof}

Notice that the distribution of $\sum_{n=1}^{\infty}(\frac{\varepsilon_{2n+1}}{n}-\frac{\varepsilon_{2n}}{n+1})$ appears in Diaconis et al.\ \cite{diac}.

\subsection{Distribution  of $\xi_1(\infty)$} 

We denote the Bessel process  of dimension 3 with drift $\nu  \geq 0$ by $\rho^{(\nu)}$. It is the norm of a Brownian motion in $\R^3$ with a drift of length $\nu$.  
\begin{coro}\label{xi1}  $\xi_1(\infty)$ has the same distribution as 
$\sup_{t \geq 0} (\varrho^{(1-\mu)}_t-t).$
\end{coro}
\begin{proof}
By Pitman and Rogers \cite{pitmanrogers}, $\varrho^{(1-\mu)}_t$ has the same distribution as $\PP_{s_1}B^{(1-\mu)}_t$, so the claim follows from (2) of Theorem \ref{les2} when $\mu\ne 0,1$, and by continuity  when $\mu=1$. For  $\mu=0$, $\xi_1(\infty)=+\infty$. \end{proof}
When $\mu=1$, $\xi_0(\infty)=+\infty$ and
 $$\xi_1(\infty)=\frac{1}{2}\sum_{n=1}^{+\infty}\frac{\varepsilon_{2n-1}+\varepsilon_{2n}}{n^2}.$$
Its distribution is studied in Biane et al.\ \cite{bpy} where it is symbolized ${\pi^2}S_2/4$. Its Laplace transform is given by, for $\tau \geq0$, $$\E(e^{-2\tau \xi_1(\infty)})=\frac{\pi^2 \tau}{\sinh^2(\pi \sqrt{\tau})}, $$ and its distribution function is (cf.\ Table 1 in \cite{bpy}) 
$$F(x)=1+2\sum_{k=1}^{+\infty}(1-4k^2x)e^{-2k^2x}.$$  In this case the corollary is also given in
Example 20 of Salminen and Yor \cite{salmy}.

When  $\mu=1/2$,
$${\xi_1(\infty)}=\sum_{n=1}^{+\infty} \frac{ 2\varepsilon_n}{n(n+1)}.$$

\begin{prop}\label{loiundemi} When $\mu=1/2$, the Laplace transform of $\xi_1(\infty)$ is, for $\tau\geq 0$,
$${\E(e^{-\tau\xi_1(\infty)})=\frac{2\pi\tau}{\cosh(\pi \sqrt{2\tau-1/4})}},$$
and its density is 
$$\sum_{n=0}^{+\infty}(-1)^{n+1} n(n+1)(2n+1)e^{-n(n+1)x/2}.$$
In distribution,
$$\xi_1(\infty)=\sup_{n>0, i=1,2,3}\frac{\varepsilon^{(i)}_n}{n},$$
where the  $\varepsilon^{(i)}_n$ are exponential independent random variables with parameter 1.
\end{prop}
\begin{proof} 
One has
$$\E(e^{-\tau  \xi_1(\infty)})=\prod_{n=1}^{+\infty}(1+\frac{2\tau }{n(n+1)})^{-1}.$$
Using the formula  
$$\cosh{\pi x}=(1+4x^2)\prod_{n=1}^{+\infty}(1+\frac{x^2}{(n+{1}/{2})^2}),$$
one obtains
$$\E(e^{-\tau  \xi_1(\infty)})
=\frac{2\pi\tau }{\cosh(\pi \sqrt{2\tau -1/4})}.$$
Let
$$g(x)=2 \sum_{n=0}^{+\infty}(-1)^n (2n+1)e^{-\frac{x}{2}(n+\frac{1}{2})^2}.$$
Since   $1/\cosh\sqrt{2\tau }$ is the Laplace transform of  ${\pi}g(\pi^2x)/2,$  (e.g.\ \cite{bpy}), one has
$$\frac{2\tau \pi}{\cosh{\pi \sqrt {2\tau -1/4}}}=\int_0^\infty  \tau  e^{-\tau  x} e^{x/8}g(x)\, dx=\int_0^\infty   e^{-\tau  x} (e^{x/8}g(x))'\, dx,$$
by an integration by parts. Computing the derivative $(e^{x/8}g(x))'$ gives the density. By integration, the distribution function of $\xi_1(\infty)$ is 
$$\sum_{n=0}^{+\infty}(-1)^{n} (2n+1)e^{-n(n+1)x/2}$$
 which equals $\prod_{n=1}^\infty(1-e^{-nx})^3$ by a formula of Jacobi (\cite{hard}, Theorem 357).
\end{proof}

\section{Asymptotics for representations of the affine algebra $A^{(1)}_1$}\label{affine}

We will show that some of the results we have obtained have a direct interpretation in term of semi-classical limits of highest weight representations of the affine Kac Moody Lie algebra $A_1^{(1)}$. The probability measure  $\nu_\lambda^{(\mu)}$ defined in (\ref{proba_dh}) is interpreted as a kind of Duistermaat Heckman measure. 
The affine string coordinates $\xi(\infty)$ describe the asymptotic behaviour of the large weights of Kashiwara's infinity crystal  $B(\infty)$.

\subsection{The Kac-Moody algebra $A_1^{(1)}$}\label{secKac}
We  consider the affine Lie algebra $A_1^{(1)}$. For our purpose, we only need to define and consider  a realization of a real Cartan subalgebra. We introduce  two copies of $\R^3$ in duality, $$\mathfrak h_\R=\mbox{Span}_{\mathbb R}\{c,\tilde{\alpha}_1,d\},\, \, \mathfrak h_\R^*=\mbox{Span}_{\mathbb R}\{\Lambda_0,\alpha_1,\delta\},$$  
 where $c=(1,0,0),\tilde{\alpha}_1=(0,1,0),  d=(0,0,1),$ and  $\Lambda_0=(1,0,0)$, $\alpha_1=(0,2,0)$, $\delta=(0,0,1)$ in $\R^3$. We let $\tilde\alpha_0=(1,-1,0)$ and $\alpha_0=(0,-2,1)$, so that $c=\tilde \alpha_0+\tilde\alpha_1$ and  $\delta=\alpha_0+\alpha_1$. Notice that these $\alpha_0$ and $\alpha_1$ project on the ones given in Section \ref{DihtoAff} by the projection on $\mathfrak{h}_\R^*/\R \delta$, identified with  $\R\Lambda_0\oplus \R\alpha_1$. With the notations of the introduction 
 $\R\Lambda_0\oplus \R\alpha_1=\R\Lambda_0\oplus \R\Lambda_1$.
 These notations are frequently used in the litterature in this context (for instance in Kac \cite{Kac}). Usually $\alpha_0, \alpha_1$ are called  the two positive simple roots of $A_1^{(1)}$ and $\tilde \alpha_0, \tilde \alpha_1$ their coroots.
 One considers the set of integral weights 
$$P=\{\lambda\in \mathfrak h_\R^*: \lambda(\tilde\alpha_i)\in \mathbb Z, i=0,1\},$$ 
and the set of dominant integral weights 
$$P_+=\{\lambda\in \mathfrak h_\R^*: \lambda(\tilde \alpha_i)\in \mathbb N, i=0,1\}.$$ 
For  a dominant integral weight $\lambda$, the character of the irreducible  representation $V(\lambda)$ of  $A_1^{(1)}$  with highest weight $\lambda$ is defined as the formal series    
\begin{align}\label{weightdecomp} 
\mbox{char}_\lambda=\sum_{\beta\in P}\mbox{dim} (V_\beta(\lambda))e^{\beta}, 
\end{align}
where $V_\beta(\lambda)$ is the weight space  corresponding to the weight $\beta$. If we let $e^\beta(h)=e^{  \beta(h)  }$ for $h\in \mathfrak{h}_\R$, and evaluate this formal series at $h$, the series converges absolutely or  diverges, and it converges when $\delta(h)>0$.  When  $\lambda=n\Lambda_0+m\frac{\alpha_1}{2}$, with $(m,n)\in \N^2$ such that $0\le m\le n$, $a\in \mathbb R$, and $b>0$, the Weyl-Kac character formula for $A_1^{(1)}$ is
\begin{align}\label{chara}
\mbox{char}_{\lambda}(a\tilde\alpha_1+bd) =\frac{\sum_{k\in \mathbb Z}\sinh(a(m+1)+2ak(n+2))e^{-b(k(m+1)+k^2(n+2))}}{\sum_{k\in \mathbb Z}\sinh(a+4ak)e^{-b(k+2k^2)}}.
\end{align}
 For more details about affine Lie algebras and their representations, see Kac \cite{Kac}.

\subsection{An analogue of the  Duistermaat Heckman measure}\label{DH}   Let  $$\lambda=t\Lambda_0+x \alpha_1/2$$ where $x\in (0,t)$ and  $\{\lambda_r, r\in \N^*\}$ be a sequence of dominant integral weights such that $$\lim_{r\to +\infty}\lambda_r/r=\lambda.$$ We consider the irreducible module $V(\lambda_r)$ of highest weight $\lambda_r$. Since  $V(\lambda_r)$ is infinite dimensional, we put a Boltzman factor on each weight.
More precisely 
 let $h_r^{(\mu)}=\frac{1}{r}(\mu\tilde\alpha_1+2d).$ As before $0 \leq \mu \leq 1$. We denote by $\gamma^{(\mu)}_r$ the probability measure on $\mathfrak h^*_\R$  given by $$\gamma^{(\mu)}_r=\frac{1}{\mbox{char}_{\lambda_r}(h_r^{(\mu)})}\sum_{\beta\in P} \dim(V_\beta(\lambda_r))e^{\beta(h_r^{(\mu)})}\,\varepsilon_{r^{-1}\beta},$$
where $\varepsilon_{r^{-1}\beta}$ is the Dirac measure at ${r^{-1}\beta}$,
\begin{prop} 
\label{propDH}  The sequence of   the push forward probabilities of $\{\gamma^{(\mu)}_r,r\ge 1\}$  by the quotient map from $\mathfrak{h}_\R^*$ to $\mathfrak{h}_\R^*/\R \delta$  converges to $\nu_\lambda^{(\mu)}$, when $r \to +\infty$.
\end{prop}

\begin{proof}
For any $v\in \mathfrak h_\R$, 
$$\int_{\mathfrak h^*_\R} e^{\beta(v)}\, \gamma^{(\mu)}_r(d\beta)=\frac{\mbox{char}_{\lambda_r}({r^{-1}v}+ h_r^{(\mu)})}{\mbox{char}_{\lambda_r}(h_r)}.$$  For $\tau \in \R$ and $v=\tau\tilde\alpha_1$,  the numerator of the   character  $\mbox{char}_{\lambda_r}({r^{-1}v}+ h_r^{(\mu)})$  given by (\ref{chara}) converges to 
$$e^{(\tau+\mu)x}\sin({\pi(\tau+\mu)}{})\varphi_{{(\tau+\mu)}{}}({t},{x})$$
when $r$ goes to infinity. Besides Lemma \ref{Poisson}  implies that  the denominator of this character is equivalent to
$$\frac{\sqrt{\pi r}}{2}e^{-\frac{1}{4}r\pi^2}\sin({\pi(\tau+\mu)}).$$
Therefore $$\lim_{r \to +\infty} \int_{\mathfrak h^*_\R} e^{\tau\beta(\tilde\alpha_1)}\, \gamma^{(\mu)}_r(d\beta)=e^{\tau x}\frac{\varphi_{\tau+\mu}(t,x)}{\varphi_{\mu}(t,x)}.$$
Then the proposition follows  from Theorem \ref{DHdrift}.
\end{proof}

The Duistermaat Heckman measure for a compact connected Lie group is  an approximation of the distribution of the weights of an irreducible finite dimensional representation when its highest weight  is large (see Heckman \cite{hec}). So the theorem above shows  that $\nu_\lambda^{(\mu)}$ is an analogue of a normalized DH-measure.  
For a compact connected Lie group,  the Duistermaat Heckman  measure also appears as  the image of the Liouville measure on a coadjoint orbit by the projection on the dual of a Cartan subalgebra. 
Frenkel has shown in  \cite{fre} that the distribution of a Brownian motion on $\mathfrak{su}(2)$ indexed by the time in $[0,1]$, given the conjugacy class of the endpoint of its wrapping on $SU(2)$, plays the role of a normalized Liouville measure on a coadjoint orbit of the loop group $L(SU(2))$. The image of this measure by the moment map associated to the action of a maximal torus  of $SU(2)$ on the coadjoint orbit is the measure $\nu_\lambda^{(\mu)} $ when $\mu=0$ (see Frenkel \cite{fre} and also Defosseux \cite{defo3}).

\subsection{Asymptotics for the crystal $B(\infty)$ of $A_1^{(1)}$} \label{binfini} The infinity crystal $B(\infty)$ of Kashiwara \cite{kash93} is the crystal  of  the Verma module with highest weight $0$ of $A_1^{(1)}$. It gives a combinatorial graph describing precisely this module.  This crystal is important since any irreducible highest weight crystal may be obtained from $B(\infty)$. It is shown in  Kashiwara \cite{kash93}, and more explicitely in Nakashima and Zelevinski \cite{nakash}, that using string parametrizations, a realization of $B(\infty)$   is given by 
\begin{align*}B(\infty)=\{x\in \N^{\N}; \mbox{ for some } n\in \N,   \frac{x_1}{1} \geq \frac{x_2}{2} \geq \cdots  \geq 
\frac{x_{n}}{n}>0, x_k=0 \mbox{ for } k > n\}.\end{align*}
Notice that the only condition on $x_0$ is $x_0 \in \N$.
For $x\in B(\infty)$, we let $$\omega(x)= \sum_{k=0}^{+\infty}x_k\alpha_k \in \mathfrak h_\R^*.$$   By  \cite{nakash}, $-\omega(x)$ is the weight of $x$ in the crystal $B(\infty)$. For
$\tilde\rho=\tilde \alpha_1/2+2d$, we let
$$s(x)=\sum_{k=0}^{+\infty}x_k=\omega(x)(\tilde\rho).$$
The character $\mbox{char}_{\infty}$ of  the Verma module of highest weight $0$ is defined as in (\ref{weightdecomp})
and is given by (see Kac \cite{Kac}, (9.7.2)) 
\begin{align}\label{VermaProd}
\mbox{char}_\infty=\prod_{\beta\in R_+}(1-e^{-\beta})^{-1},
\end{align} where \begin{align}\label{rac_pos}R_+=\{\alpha_0+n\delta, \alpha_1+n\delta, (n+1)\delta, n\in \N\} \end{align} is the set of  positive roots of $A_1^{(1)}$.
 As previously, if we let $e^\beta(h)=e^{  \beta(h)  }$ for $h\in \mathfrak{h}_\R$, and evaluate the  formal character at $h$, it converges when $\delta(h)>0$. 
Let $r\in\R_+^*$. 
On each element $x$ of the crystal $B(\infty)$ we put the Boltzman weight  $e^{-s(x)/r}$. 
 We introduce the probability distribution $\beta_r$ on $B(\infty)$  by 
$$\beta_r=\frac{1}{Z_r}\sum_{x \in B(\infty)}e^{-s(x)/r}\varepsilon_x$$
when $\varepsilon_x$ is the Dirac mass at $x$ and (by Kashiwara \cite{kash93})
$$Z_r=\sum_{x\in B(\infty)}e^{-s(x)/r}={\mbox{char}_\infty(\tilde\rho/r)}.$$
The following theorem indicates that the affine Brownian model describes a  kind of continuous version of the infinity crystal $B(\infty)$ for the affine Lie algebra  $A_1^{(1)}$.  Let $j_r:B(\infty) \to \R^{\N}$ be given by 
$j_r((x_k, k \geq 0))=(x_k/r, k \geq 0)$.

\begin{theo} The image of $\beta_r$ by the map $j_r$ converges in distribution to the distribution of the Verma parameter $\xi(\infty)$ of $B^{(1/2)}$, when $r\to +\infty$.
\end{theo}
\begin{proof} We first notice that 
   $$A_\infty=\{(\lambda_1,\lambda_2,\cdots); (0, \lambda_1,\lambda_2,\cdots) \in B(\infty)\}$$
is  the set of anti-lecture hall compositions defined in Corteel and Savage \cite{cort1}.
 For $\lambda=(\lambda_1,\lambda_2\cdots, \lambda_n,0,0,\cdots)\in A_\infty$, let
 $|\lambda|=\sum_{k=1}^{n} \lambda_k$.  Then, see Corteel et al. \cite{cort1,cort} and (1.2) and (1.3) in Chen et al.\  \cite{chen}, for $0 \leq q < 1, k \in \N$,
\begin{align}\label{suma}\sum_{\lambda \in A_\infty}q^{|\lambda|}&=\frac{(-q;q)_\infty}{(q^2;q)_\infty},\\\label{sumb}
\sum_{\lambda \in A_\infty, \lambda_1\leq k}q^{|\lambda|}&=\frac{(-q;q)_\infty(q;q^{k+2})_\infty(q^{k+1};q^{k+2})_\infty(q^{k+2};q^{k+2})_\infty}{(q;q)_\infty},\end{align} where $(a;q)_\infty=\prod_{n=0}^\infty(1-aq^n)$. 

Let $X^{(r)}=(X_k^{(r)}; k \geq 0)$ be a random element with distribution $\beta_r$. First it is  clear that $X^{(r)}_0/r$ converges in distribution to $\xi_0(\infty)$. Let us now prove that $X^{(r)}_1/r$ converges in distribution to $\xi_1(\infty)$. It follows from (\ref{suma}) and (\ref{sumb}) that, for $a \geq 0$ and $q=e^{-1/r}$,
\begin{align*}\P(X^{(r)}_1 \leq a r)&=\frac{\sum_{\{x\in B(\infty); x_1 \leq [ar]\}}q^{s(x)}}{\sum_{\{x\in B(\infty)\}}q^{s(x)}} =\frac{\sum_{\{\lambda\in A_\infty; \lambda_1 \leq [ar]\}}q^{|\lambda|}}{\sum_{\{\lambda\in A_\infty\}}q^{|\lambda|}}\\
=&\frac{1}{1-q} (q;e^{- [ar]/r}q^2)_\infty(e^{-[ar]/r}q;e^{- [ar]/r}q^2)_\infty(e^{-[ar]/r}q^2;e^{- [ar]/r}q^2)_\infty.
\end{align*}
Since
$(q;e^{- [ar]/r}q^2)_\infty$ is equivalent to
$(1-q)(e^{- a};e^{- a})_\infty$ when $q$ tends to $1$,
one has, by Proposition \ref{loiundemi},
$$\lim_{r \to \infty} \P({X^{(r)}_1} \leq r a)= (e^{- a};e^{- a})_\infty^3=\prod_{n=1}^{\infty}(1-e^{-na})^3=\P(\xi_1(\infty) \leq a).$$
So $X^{(r)}_1/{r}$ converges to $\xi_1(\infty)$.
We now consider the full sequence ${X^{(r)}_k}/{r}, k\in \N$. 
For any $r$ and $n \geq 1$, one has
$$\frac{X^{(r)}_1}{1}\ge \frac{X^{(r)}_2}{2}\ge\cdots\ge \frac{X^{(r)}_n}{n}\ge 0,$$ 
which implies that for any  $n\in \N$, the collection of   distributions of $$(\frac{1}{r}X^{(r)}_1,\cdots,\frac{1}{r}X^{(r)}_n)_{r>0}$$ is tight since $X^{(r)}_1/{r}$ converges in distribution. By Cantor's diagonal argument, we construct an increasing sequence $\varphi(r)\in \N, r \in \N,$ such that  the random variables $\frac{1}{\varphi(r)}X_k^{(\varphi(r))},k\ge 0,$ converge in finite dimensional distribution when $r$ goes to infinity. Let us denote by $(R_k,k\ge 0)$ the limit, and let us prove that
$\frac{R_k}{k}$ has the same distribution as $\sum_{n=k}^{+\infty}\frac{2\varepsilon_n}{n(n+1)}$ where the $\varepsilon_n$'s are independent exponential random variables with parameter 1.
For $x\in B(\infty)$, one has, since $s(x)=x_0+\frac{1}{2}\sum_{k=2}^{+\infty}(kx_{k-1}-(k-1)x_k)$,
\begin{align*}
&\P(X^{(r)}=x)=\frac{e^{-\frac{1}{r}x_0}}{\mbox{char}_\infty({r^{-1}}\tilde\rho)}\prod_{k=2}^\infty e^{-\frac{1}{2r}(kx_{k-1}-(k-1)x_k)}1_{\{\frac{x_{k-1}}{k-1}\ge \frac{x_k}{k}\}}\\
&=\frac{e^{-\frac{1}{r}x_0}}{\mbox{char}_\infty({r^{-1}}\tilde\rho)}\prod_{k=1}^\infty e^{-\frac{k+1}{2r}(x_{k}-\lceil \frac{k}{k+1}x_{k+1}\rceil)}e^{-\frac{1}{2r}(k\lceil \frac{k-1}{k}x_k\rceil-(k-1)x_k)}1_{\{x_{k}\ge \lceil k\frac{x_{k+1}}{k+1}\rceil\}},
\end{align*} where $\lceil  \cdot \rceil$ is the ceiling function. Let  $Y_k^{(r)},$ $ k=0,\cdots, n$, be independent geometric random variables  with values in $\N$, where $Y_0^{(r)}$ has the parameter $1$ and $Y_k^{(r)}$ has the parameter $e^{-\frac{k+1}{2r}}$ when $k\geq 1$.
As, 
$$e^{-\frac{k}{2r}}\le e^{-\frac{1}{2r}(k\lceil \frac{k-1}{k}x_k\rceil -(k-1)x_k)}\le 1,$$
one obtains that for $t_0,\cdots, t_n\ \geq0 $,
\begin{align*}C(r,n)(1+O(\frac{1}{r})) \le \frac{\P(X^{(r)}_0\leq t_0,\cdots, X_n^{(r)}-\lceil \frac{n}{n+1}X_{n+1}^{(r)}\rceil\leq t_n)}{\P(Y_0^{(r)}\leq t_0,\cdots ,Y^{(r)}_n\leq t_n)} \le C(r,n),\end{align*}
where $C(r,n)$ is  independent of $t_0,\cdots,t_n$, and tends to 1 when $r$ tends to $ +\infty$. This proves that for any $n$, 
$$\frac{1}{r}X_0^{(r)},\,\frac{1}{r}(X_k^{(r)}-\lceil\frac{k}{k+1}X_{k+1}^{(r)}\rceil),\quad 1 \leq k\leq n ,$$
converge jointly to 
 $\varepsilon_0,\, \frac{2}{k+1}\varepsilon_k, 1 \leq k \leq n,$ when $r$ goes to infinity.
Besides $$\lim_{r \to \infty}\frac{1}{\varphi(r)}(X_k^{(\varphi(r))}-\lceil\frac{k}{k+1}X_{k+1}^{(\varphi(r))}\rceil) =R_k- \frac{k}{k+1}R_{k+1}.$$  Thus, for $k\ge 1$,  
 $R_k- \frac{k}{k+1}R_{k+1}$ are independent random variables with the same distribution as $\frac{2}{k+1}\varepsilon_k$. The positive sequence
 $R_k/k$ is  decreasing. Let $S$ be its limit.  We have the identity in distribution, for all $k \geq 1$,
$$\frac{R_k}{k}=\sum_{n=k}^{+\infty}\frac{2\varepsilon_n}{n(n+1)}+S.$$
We have proved that in distribution $R_1=\xi_1(\infty)$, so $S=0$ which finishes the proof. \end{proof} 

Recall (Corollary \ref{Dundemi}) that $L^{(1/2)}(\infty)=D^{1/2}(\infty)\alpha_1/2$ where the density of $D^{1/2}(\infty)$ is
$1/(\pi \cosh x)$, and that $-\omega(x)$ is the weight of $x\in B(\infty)$.

\begin{prop} When $r$ goes to infinity, in distribution, if $X^{(r)}$ has the distribution $\beta_r$,
\begin{enumerate}
\item in the quotient space $\mathfrak{h}_\R^*/\R\delta$, the normalized weights $\omega(X^{(r)})/r$ converge  to $L^{(1/2)}(\infty)$,
\item the coordinate of $\omega(X^{(r)})/r$ along $\delta$ goes to $+\infty$.
\end{enumerate}
\end{prop}
\begin{proof} One has  for any $u\in \mathfrak h_\R$, 
$$\E(e^{-\omega(X^{(r)})(u)})=\frac{\mbox{char}_\infty(u+r^{-1}\tilde\rho)}{ \mbox{char}_\infty(r^{-1}\tilde\rho)}.$$
In view of (\ref{rac_pos}),  the  expression (\ref{VermaProd}) of the character gives the Laplace transform of $\omega(X^{(r)})$ and shows that, in distribution,  
$$\omega(X^{(r)})=\sum_{n\ge 0}(G_0(n)(\alpha_0+n\delta)+G_1(n)(\alpha_1+n\delta)+G_2(n)(n+1)\delta).$$
Here $G_i(n)$, $i=0,1,2$, $n\in \N$, are independent  random variables such that $G_0(n)$, $G_1(n)$ and $G_2(n)$, are geometrically distributed with respective  parameter $e^{-(\alpha_0+n\delta)(r^{-1}\tilde\rho)}$, $e^{-(\alpha_1+n\delta)(r^{-1}\tilde\rho)}$, and $e^{-(n+1)\delta(r^{-1}\tilde\rho)}$, i.e.\  $e^{-2(n+1/2)/r}$, $e^{-2(n+1/2)/r}$ and $e^{-2(n+1)/r}$.  The proposition follows easily.  
\end{proof}

Due to the L\'evy correction, the image by $\omega$ of the limit in law of $\frac{X^{(r)}}{r}$ is not equal to limit in law of $\omega( \frac{X^{(r)}}{r})$, modulo $\delta$.

\end{document}